\documentclass[12pt, reqno, twoside, letterpaper]{amsart}

\usepackage{paperstyle}
\usepackage{graphicx}

\usepackage{mathtools}
\usepackage{todonotes}

\usepackage{tikz-cd}

\newcommand{\be}{\begin{equation}}
\newcommand{\ee}{\end{equation}}
\newcommand{\dalign}[1]{\[\begin{aligned} #1 \end{aligned}\]}
\newcommand{\nearint}[1]{\left\llbracket #1 \right\rrbracket}
\newcommand{\borels}[2]{{\bf\Sigma}^{#1}_{#2}}
\newcommand{\borelp}[2]{{\bf\Pi}^{#1}_{#2}}
\newcommand{\boreld}[2]{{\bf\Delta}^{#1}_{#2}}
\newcommand{\squig}[1]{\,\mathop\rightsquigarrow\limits^{#1}\,}
\newcommand{\RQ}{{\R\!\setminus\!\Q}}
\def\ttt{{\tt t}}

\def\gell{{
	{\frak L}
}}
\def\gB{{
	{\frak B}
}}
\def\gb{{
	{\frak b}
}}

\title[Descriptive properties of the type of an irrational number]
{Descriptive properties of\\ the type of an irrational number}
      
\author[W.~Banks]{William Banks}

\address{University of Missouri, Columbia, USA.}

\email{bankswd@missouri.edu}
    
\author[A.~Harcharras]{Asma Harcharras}

\address{University of Missouri, Columbia, USA.}

\email{harcharrasa@missouri.edu}
    
\author[D.~Lecomte]{Dominique Lecomte}

\address{Sorbonne Universit\'e,
Institut de Math\'ematiques de Jussieu-Paris Rive Gauche, 
Paris, France and Universit\'e de Picardie, IUT de l'Oise,
Creil, France.}

\email{dominique.lecomte@upmc.fr}
    
\date{\today}

\begin{document}

\begin{abstract}
The type $\tau(\alpha)$ of an irrational number $\alpha$
measures the extent to which rational numbers can closely approximate $\alpha$.
 More precisely, $\tau(\alpha)$ is the
infimum over those $t\in\R$ for which $|\alpha-h/k|<k^{-t-1}$
has at most finitely many solutions $h,k\in\Z$, $k>0$.
In this paper, we regard the type as a function $\tau:\RQ\to[1,\infty]$
and explore its descriptive properties. We show that $\tau$ is invariant under
the natural action of $GL_2(\Q)$ on $\RQ$.
We show that $\tau$ is densely onto, and we compute the descriptive
complexity of the pre-image of the singletons and of certain intervals.
Finally, we show that the function $\tau$ is $[1,\infty ]$-upper
semi-Baire class 1 complete.
\end{abstract}

\thanks{MSC Primary: 03E15, 11J82; Secondary: 26A21, 11J70.}


\maketitle

\newpage

\tableofcontents

%%%%%%%%%%%%%%%%%%%%%%%%
%%%%% PAPER BEGINS %%%%%
%%%%%%%%%%%%%%%%%%%%%%%%

\newpage{\large\section{Introduction}}

\emph{Descriptive set theory} is a branch of mathematical logic that revolves around
the study of ``definable sets'' in Polish (i.e., separable, completely metrizable)
spaces.  In this theory, well-behaved sets are classified in
hierarchies  arranged according to the complexity of their
definitions, and the structure of the sets at each level in the hierarchy
is methodically analyzed. Descriptive set theory is
a primary area of research in set theory and has
applications in other areas of mathematics, including ergodic theory, functional
analysis, and the study of operator algebras and group actions.

\emph{Number theory} (or higher arithmetic) is a branch of pure mathematics
that is devoted to the study of prime numbers, integers, integer-valued functions, 
and mathematical objects constructed from the integers (for example, the
set $\Q$ of rational numbers). One also studies real numbers in relation to rational numbers, for example, as approximated by the latter; this is Diophantine approximation.
Although many problems in number theory can be approached using
analytic or algebraic techniques, such methods do not lend themselves
well to the study of irrational numbers.
Transcendental number theory bears little resemblance
to other branches of number theory, although it is
an indispensable part of the field.

The present paper is a rare example in the literature of results that
bridge the two branches of mathematics described above (another example is
the excellent paper of Jackson et al~\cite{Jackson},
which contains an application
of descriptive set theory to number theory).
Our work originates in the observation that \emph{the set
of irrational real numbers $\RQ$ is a Polish space} (in the usual topology),
and hence various distinguished subsets of $\RQ$
are of simultaneous interest in both disciplines.
The initial aim of this paper was to
combine techniques from both descriptive set theory and number theory
to determine precisely the descriptive complexity of subsets of $\RQ$ consisting of
all irrational numbers of a specified \emph{irrationality type}
(see the definition in \S\ref{sec:irrtypdef}).
Later on, as our tools developed, our goals became more ambitious. In the end,
we have shown (among other things and in a sense to be made precise below) that the
function afforded by irrationality type has the highest descriptive complexity
in its class; it is likely the first concrete example of a \emph{Baire complete}
function for the natural class of functions in which it resides.

Our methods are fairly general and can be applied to other Polish
spaces. To formulate our results, we first recall some standard terminology.\footnote{Due
to the interdisciplinary nature of this work, we have included many
basic definitions throughout the paper to keep the exposition fairly
self-contained and accessible to readers from both areas of mathematics.}

\subsection{The type of an irrational number}
\label{sec:irrtypdef}

For a real number $x$, we write $\nearint{x}$
to denote the distance from $x$ to the nearest integer:
$$
\nearint{x}\defeq\min\limits_{n\in\Z}|x-n|.
$$
For any irrational number $\alpha$, the \emph{type} of $\alpha$
is the quantity defined by\footnote{This definition of the type is equivalent
to the definition given in our abstract; proof of the equivalence is
given in the appendix.}
\be\label{eq:tau-defn}
\tau(\alpha)\defeq\sup\big\{\theta\in\R:
\mathop{\underline{\rm lim}}\limits_{q\in\N}
~q^\theta\nearint{q\alpha}=0\big\}.
\ee
The Dirichlet approximation theorem\footnote{Dirichlet's approximation
theorem asserts that for any $\alpha,Q\in\R$, $Q\ge 1$, there is a rational
number $p/q$ with $1\le q\le Q$ such that $|\alpha-p/q|<1/(qQ)$;
see, e.g., Bugeaud~\cite[Thm~1.1]{Bugeaud}. From this, it follows
that for any irrational $\alpha$
one has $q\nearint{q\alpha}<1$ for infinitely many $q\in\N$.}
implies that $\tau(\alpha)\in[1,\infty]$. One says that $\alpha$ is
of \emph{finite type} if $\tau(\alpha)<\infty$.
The celebrated theorems of Khinchin~\cite{Khin1} and of 
Roth~\cite{Roth1, Roth2} assert that $\tau(\alpha)=1$ for almost
all real numbers (in the sense of the Lebesgue measure) and all
irrational algebraic numbers, respectively.

A \emph{Liouville number} is an irrational number $\alpha$ with the
property that for each positive integer~$n$ there exist integers $p$
and $q>1$ such that
\be\label{eq:liouville-ineq}
\bigg|\alpha-\frac{p}{q}\bigg|<\frac{1}{q^n}.
\ee
Using \eqref{eq:liouville-ineq} it is easy to show
that $\alpha$ is a Liouville number if and only if $\tau(\alpha)=\infty$.

We remark that the quantity $\mu(\alpha)\defeq\tau(\alpha)+1$
is called the \emph{irrationality measure} (or the \emph{Liouville-Roth constant})
associated with $\alpha$, and very often it occurs in the literature instead of
$\tau(\alpha)$; this is largely a matter of taste.

\subsection{Borel hierarchy and completeness of sets}
\label{sec:borelhierarchy}
Let $\cX$ be a Polish space, i.e., a separable and completely metrizable
topological space. The \emph{Borel hierarchy} on $\cX$ consists of 
sets $\borels0\xi (\cX)$, $\borelp0\xi (\cX)$, and $\boreld0\xi (\cX)$,
which are defined for every countable ordinal $\xi>0$. The elements of these sets are all
subsets of $\cX$. The space $\cX$ is not mentioned when there is no ambiguity
or when $\cX$ is considered as a variable (in the latter case, for instance,
a set $A$ is in the \emph{class} $\borels0\xi$ if there is a Polish space $\cX$
such that $A$ belongs to the \emph{set} $\borels0\xi (\cX)$).
When $\cX$ is fixed, these sets are defined inductively according to the following rules:
\begin{itemize}
\item[$\bullet$] $A\in\borels01$ if and only if $A$ is open in $\cX$;
\item[$\bullet$] $A\in\borelp0\xi$ if and only if
$\cX\setminus A\in\borels0\xi$;
\item[$\bullet$] For $\xi>1$, $A\in\borels0\xi$ if and only if
$A$ is the union of some countable collection
$\{A_j:j\in\N\}$ such that each $A_j$ lies in $\borelp0{\xi_j}$
for some $\xi_j<\xi$;
\item[$\bullet$] $A\in\boreld0\xi$ if and only if 
$A\in\borels0\xi$ and $A\in\borelp0\xi$.
\end{itemize}
The class of Borel sets ramifies in the following hierarchy:
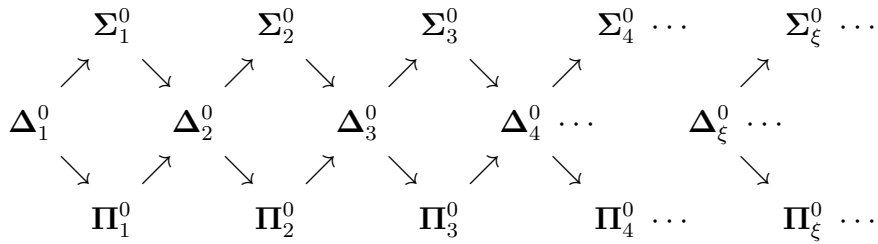
\begin{alignat*}{50}
&&\borels01&&&&
\borels02&&&&
\borels03&&&&
\borels04&~~\cdots&&&
\borels0\xi&~~\cdots\\
&\nearrow&&\searrow&&\nearrow&&\searrow&
&\nearrow&&\searrow&&\nearrow&&  &&\nearrow&\\
\boreld01&&&&
\boreld02&&&&
\boreld03&&&&
\boreld04&~~\cdots&&&
\boreld0\xi&~~\cdots&&&\\
&\searrow&&\nearrow&&\searrow&&\nearrow&
&\searrow&&\nearrow&&\searrow&&  &&\searrow&\\
&&\borelp01&&&&
\borelp02&&&&
\borelp03&&&&
\borelp04&~~\cdots&&&
\borelp0\xi&~~\cdots
\end{alignat*}
and every class is contained in any class to the right of it.
If $\cX$ is an uncountable Polish space, then one has
$\borels0\xi\ne\borelp0\xi$ for every countable ordinal $\xi>0$;
for a proof, see Kechris~\cite[Thm.~22.4]{K}.

Note that $\RQ$ is Polish since it is a $G_\delta$ subspace
of the Polish space~$\R$ (see \cite[Thm.~3.11]{K}). A subset of a topological
space $\cX$ is \emph{clopen} if it is closed and open, and $\cX$ is \emph{zero-dimensional} if it has a basis consisting of clopen sets. The space $\RQ$ 
is zero-dimensional since it is the complement of~$\Q$, which
is dense in $\R$. Moreover, $\RQ$ is uncountable, hence by \cite[Thm.~22.4]{K} we have
$$
\borels0\xi(\RQ)\ne\borelp0\xi(\RQ)
$$
for every countable ordinal $\xi>0$.

If $\cX$ and $\cY$ are topological spaces, $A\subset \cX$, and
$B\subset \cY$, then $A$ is said to be
\emph{Wadge reducible to $B$} if there is a continuous map $f:\cX\to \cY$
with $f^{-1}(B)=A$ (that is, $x\in A\Longleftrightarrow f(x)\in B$);
in this situation, we write $A\le_W B$. Wadge reducibility
provides a notion of the relative complexity of sets in topological
spaces, where the notation $A\le_W B$ indicates that $A$ is
``simpler'' than $B$ in a suitable sense. The relation $\le_W$
is reflexive and transitive; it imposes an
(essentially \text{well-ordered} in zero-dimensional spaces) hierarchy on the Borel sets,
called the \emph{Wadge hierarchy}.
The classes $\borels0\xi$ and $\borelp0\xi$ are closed under
continuous pre-images, and these classes are
initial segments in the Wadge hierarchy.

Let $\bf\Gamma$ be a class of sets defined on Polish spaces.
If $\cY$ is a Polish space, a subset $B\subset\cY$
is called \emph{$\bf\Gamma$-complete} if $B\in {\bf\Gamma}(\cY)$
and $A\le_W B$ for any $A\in {\bf\Gamma}(\cX)$, where 
$\cX$ is any arbitrary zero-dimensional Polish space; see \cite[Def.~22.9]{K}.
In particular, if $\cY$ is a zero-dimensional Polish space, then by
a theorem of Wadge, one knows that $B\subset\cY$ is
$\borels0\xi$-complete [resp.\ \text{$\borelp0\xi$-complete}]
if and only if $B$ belongs to $\borels0\xi\setminus\borelp0\xi$
[resp.\ $\borelp0\xi\setminus\borels0\xi$]; see, for example,
\cite[Thm.~22.10]{K}. In other words, relative to the 
Wadge ordering $\le_W$, the sets in
$\borels0\xi\setminus\borelp0\xi$ are maximal among all
$\borels0\xi$ sets
(and similarly switching $\borels0\xi$ and $\borelp0\xi$).
We remark that the hypothesis that $\cX$ is zero-dimensional
guarantees the existence of ``sufficiently many'' continuous functions
(by contrast, the only continuous functions from
$\R$ into the Cantor space $\{0,1\}^\omega$ are the constant 
functions). On the other hand, it is unnecessary to assume that $\cY$ is
zero-dimensional in Wadge's theorem, for it holds in any uncountable
Polish space; see \cite[24.20]{K}.

\subsection{Baire hierarchy and completeness of functions}
There is also a hierarchy of functions between Polish spaces called the
\emph{Baire hierarchy}. Recall that a function is continuous if the pre-image
of any open set is open, i.e., in $\borels{0}{1}$. A function is called
\emph{Baire class~1} if the pre-image of any open set is in $\borels{0}{2}$. More
generally, in the light of \cite[24.1 and 24.3]{K}, a function is called
\emph{Baire class $\xi$} if the pre-image of any open set is in $\borels{0}{\xi+1}$, 
for any countable ordinal $\xi$. Thus, the Baire class 0 functions are 
the continuous ones.

Recall that a function $f$ from a Polish space $\cX$ into 
$\overline\R\defeq\R\cup\{ -\infty ,\infty\}$ is said to be \emph{upper semi-
continuous} (resp., \emph{lower semi-continuous}) if $f^{-1}([-\infty ,r))$ (resp., 
$f^{-1}((r,\infty ])$) is an open subset of $\cX$ for each real number $r$. Considering 
the usual order topology on $\overline\R$, which is homeomorphic to $[-1,1]$ and thus 
Polish, we see that the semi-continuous functions are (relatively simple) Baire class~1 
functions. We naturally extend these notions to countable ordinals by saying that a 
function $f$ from a Polish space $\cX$ into a subset 
$K\subset\overline\R$ is \emph{upper semi-Baire class $\xi$}
[resp.\ \emph{lower semi-Baire class $\xi$}] if 
$f^{-1}([-\infty ,r))$ \big[resp.\ $f^{-1}((r,\infty ])$\big] is a $\borels{0}{\xi +1}$ 
subset of $\cX$ for every real number $r$. 
As we show in Lemma \ref{eq:norah-jones}, there is an equivalent definition of the type
of an irrational number in terms of limit superior.
According to Elekes et al~\cite[Theorem 2.1]{Elekes}, 
functions defined in terms of limit superior
are deeply connected with the upper-semi Baire class 1 functions.

The notion of a complete $K$-upper semi-continuous function was
introduced and characterized in Solecki~\cite[Section 5]{Solecki}; this notion naturally 
extends here. Let $f:\cX\to K$ be an upper semi-Baire class $\xi$ function.
We say that $f$ is \emph{$K$-upper semi-Baire class $\xi$ complete} if
for every upper semi-Baire class $\xi$ function $g:2^\omega\to K$, there 
exists a continuous function $\phi:2^\omega\to\cX$ such that $g=f\circ\phi$.
This notion generalizes the notion of completeness of sets defined
above.\footnote{We emphasize that the target space $K$ is an essential part of the
definition. For example, if $f,g$ are constant functions taking different 
values, no function $\phi$ has the stated property.}
Indeed, let $\ind{A}:\cX\to\{ 0,1\}$ be the characteristic function
a given subset $A\subset\cX$. Then $\ind{A}$ is upper semi-Baire class $\xi$ if and
only if $A\in\borelp{0}{\xi +1}(\cX)$, and
$\ind{A}$ is $\{ 0,1\}$-upper semi-Baire class $\xi$ complete if and only if
$A$ is a $\borelp{0}{\xi+1}$-complete subset of $\cX$.

\subsection{Statement of results}

The group ${\rm GL}_2(\Q)$ acts naturally on the space $\RQ$
via \emph{M\"obius transformations}:
\be\label{eq:Mobius}
g \alpha\defeq\frac{a\alpha+b}{c\alpha+d}
\quad\text{for any~}g=\begin{pmatrix}a&b\\c&d\end{pmatrix}\in {\rm GL}_2(\Q)
\text{~and~}\alpha\in\RQ.
\ee
Our first theorem (proved in \S\ref{sec:GL2Q}) asserts that the type
of an irrational number is well-defined on the orbits of this action.

\begin{theorem}\label{thm:main-GL2}
For all $\alpha\in\RQ$ and $g\in{\rm GL}_2(\Q)$ we have
$\tau(g\alpha)=\tau(\alpha)$.
\end{theorem}

A consequence is that the type of an irrational number
does not depend on the finite beginning of its continued fraction development
(see Corollary~\ref{cor:get-off-my-tail}). Next, we show that $\tau:\RQ\to[1,\infty]$ is a
\emph{densely onto} map, meaning that the pre-image of any singleton is dense
in the domain. Hence, the map afforded by type is
discontinuous in quite a spectacular way.

\begin{theorem}\label{thm:densely-onto}
The function $\tau:\RQ\to[1,\infty]$ is surjective.
Moreover, for any given $\ttt\in[1,\infty]$, the set
$\{\alpha\in\RQ:\tau(\alpha)=\ttt\}$
is dense in the space $\RQ$.
\end{theorem}

We give three proofs of Theorem~\ref{thm:densely-onto}, one
in \S\ref{sec:GL2Q}, another in \S\ref{sec:heads-or-tails}, and another in \S\ref{sec:complete functions}. The next result is proved in \S\ref{sec:complexity}.

\begin{theorem}\label{thm:main2}
For every $\ttt\in[1,\infty)$, the set $\{\alpha\in\RQ:\tau(\alpha)=\ttt\}$
is $\borelp{0}{3}$-complete. In other words, it is a $\borelp03$ set
that does not belong to $\borels03$. The set $\{\alpha\in\RQ:\tau(\alpha)=\infty\}$ is $\borelp{0}{2}$-complete.
\end{theorem}

Our main result is the following; it is proved in \S\ref{sec:complete functions}.

\begin{theorem}\label{thm:main3}
The type function is $[1,\infty]$-upper semi-Baire class~1 complete.
\end{theorem}

Very often in descriptive set theory, we can construct complete or universal objects using ad-hoc internal methods. There are few classes of sets or functions for which natural complete elements, coming from the rest of mathematics, are known. The type function provides an example of such an object.

\newpage{\large\section{An oriented graph}}

\subsection{Edges and trails}
\label{sec:edges-and-trails}

We denote by $\N$ the set of natural numbers (including zero)
and by $\N_1$ the set of strictly positive natural numbers. In the sequel,
we extensively use an oriented graph $\cG$, defined as follows.
For the vertex set, we take
$$
V\defeq\big\{(x,y)\in\N_1^2:x<y\}.
$$
Two vertices $v=(x,y)$ and $\tilde v=(\tilde x,\tilde y)$
are connected by a directed edge from $v$ to $\tilde v$
if and only if there is a \emph{positive} integer $a$ such that
\be\label{eq:squig-games}
\tilde x=y\mand\tilde y=ay+x.
\ee
We write $v\squig{a}\tilde v$ in this situation, and
we say that $\tilde v$ is a vertex \emph{succeeding}~$v$,
and that $v$ is the vertex \emph{preceding}~$\tilde v$.
Note that every vertex $(\tilde x,\tilde y)$ has at most one predecessor
$(x,y)$, since the relations \eqref{eq:squig-games}
imply that
$$
a=\lfloor{\tilde y}/{\tilde x}\rfloor,\qquad
x=\tilde y-a\tilde x,\qquad
y=\tilde x,
$$
and therefore $(x,y)$ is determined uniquely by
$(\tilde x,\tilde y)$.\footnote{For all $u\in\R$, we denote by $\fl{u}$
the largest integer that does not exceed $u$.}

For a given vertex $v_0\in V$ and a sequence
$a=(a_n)_{n\ge 1}$ of positive integers, we denote by ${\tt Tr}(v_0,a)$ 
the associated \emph{trail}:
\be\label{eq:inf-trail}
{\tt Tr}(v_0,a):\quad v_0\squig{a_1} v_1
\squig{a_2} v_2
\squig{a_3} v_3
\squig{a_4}  \cdots.
\ee

In this paper, we are interested in studying irrational numbers
of a fixed type. Motivated by a characterization of type given 
in \S\ref{sec:type-characterization} below (see \eqref{eq:type-character})
we introduce a map $\ell:V\to [1,\infty)$ which is defined by
\be\label{eq:ell-defn}
\forall\,v=(x,y)\in V:\quad
\ell(v)\defeq\begin{cases}
{\displaystyle\frac{\log y}{\log x}}&\quad\text{if $x>1$},\\ \\
1&\quad\text{if $x=1$}.
\end{cases}
\ee
For an infinite trail
${\tt T}={\tt Tr}(v_0,a)$ of the form \eqref{eq:inf-trail},
we define the \emph{$\ell$-limit of ${\tt T}$} as
$$
\ell({\tt T})\defeq\mathop{{\rm lim}}\limits_{n\to\infty}\ell(v_n),
$$
provided this limit exists. We also define the
\emph{$\ell$-limsup of ${\tt T}$} to be the quantity
$$
\overline\ell({\tt T})
\defeq\mathop{\overline{\rm lim}}\limits_{n\to\infty}\ell(v_n).
$$
Note that $\overline\ell({\tt T})\in[1,\infty]$.

\subsection{Hitting the target}

\begin{lemma}\label{lem:targeting}
Let $\eps\in(0,1)$, and suppose that
\be\label{eq:ynu-bd}
\ttt\ge 1\mand y>(10/\eps)^{1/\eps}.
\ee
Given $v=(x,y)\in V$,
there exists a positive integer $a$ and vertex
$\tilde v=(\tilde x,\tilde y)\in V$ such that
\eqref{eq:squig-games} holds
$($in other words, $v\squig{a}\tilde v)$
and $\ell(\tilde v)\in[\ttt,\ttt+\eps)$.
\end{lemma}

\begin{proof}
For all $u\in\R$, let $\rf{u}$ be the smallest integer that is greater than
or equal to $u$; note that $u\le\rf{u}<u+1$ for all $u$. The integer
$$
a\defeq\rf{y^{\ttt-1}-x/y}
$$
is at least one since $x<y$ and $y^{\ttt-1}\ge 1$, and
we have
$$
y^{\ttt-1}-x/y\le a<y^{\ttt-1}-x/y+1.
$$
Put $\tilde x\defeq y$ and
$\tilde y\defeq ay+x$, so that \eqref{eq:squig-games} holds.
Since
$$
\tilde y=ay+x\ge(y^{\ttt-1}-x/y)y+x=y^\ttt,
$$
it follows that
$$
\ell(\tilde v)=\frac{\log\tilde y}{\log\tilde x}
\ge\frac{\log y^\ttt}{\log y}=\ttt.
$$
Similarly,
$$
\tilde y=ay+x<(y^{\ttt-1}-x/y+1)y+x=y^\ttt+y,
$$
and therefore
$$
\ell(\tilde v)=\frac{\log\tilde y}{\log\tilde x}
<\frac{\log(y^\ttt+y)}{\log y}
=\frac{\ttt\log y+\log(1+y^{1-\ttt})}{\log y}
<\ttt+\frac{y^{1-\ttt}}{\log y},
$$
where in the last step, we used the fact that $\log(1+u)\le u$
for all $u>0$. In view of \eqref{eq:ynu-bd} we have
$$
\frac{y^{1-\ttt}}{\log y}\le\frac{1}{\log y}<\frac{\eps}{\log(10/\eps)}<\eps,
$$
thus we obtain the upper bound $\ell(\tilde v)<\ttt+\eps$. Thus 
$\ell(\tilde v)\in[\ttt,\ttt+\eps)$.
\end{proof}

\begin{proposition}\label{prop:happy-trails}
For every $v=(x,y)\in V$ and $\ttt\in[1,\infty]$, there exists
an infinite trail ${\tt T}={\tt Tr}(v,a)$
whose $\ell$-limit is $\ell({\tt T})=\ttt$.
\end{proposition}

\begin{proof}
In the argument that follows, we use induction on $n$ to
construct a sequence $(a_n)_{n\in\N}\subset\N$ 
for which the resulting trail
$$
{\tt T}={\tt Tr}(v,a):\quad
v\squig{a_0}
v_0\squig{a_1} v_1
\squig{a_2} v_2
\squig{a_3} v_3
\squig{a_4}  \cdots
$$
has the desired property.

Let $v_0\defeq(x_0,y_0)$ be the vertex defined by
$x_0=y$ and $y_0=400y+x$. Put $a_0\defeq 400$, and note that
$v\squig{a_0}v_0$. 

Next, suppose the vertex $v_n=(x_n,y_n)\in V$ has been
defined for some $n\in\N$. Let $m\ge 2$ be the unique integer such that
$$
(10m)^m<y_n\le(10m+10)^{m+1}
$$
(because each $y_n>400$, such an integer $m$ must exist). Put
$$
\eps_m\defeq m^{-1}\mand\ttt_m\defeq \ttt+2\eps_m\ge 1.
$$
Since $y_n>(10/\eps_m)^{1/\eps_m}$, we can apply Lemma~\ref{lem:targeting}
to conclude that there exists a positive integer $a_{n+1}$ and a vertex
$v_{n+1}=(x_{n+1},y_{n+1})\in V$ such that
$$
v_n\squig{a_{n+1}}v_{n+1}
\mand
\ell(v_{n+1})\in[\ttt_m,\ttt_m+\eps_m)=[\ttt+2\eps_m,\ttt+3\eps_m).
$$
As $n$ tends to infinity, we have
$$
y_n\to\infty\quad\Longrightarrow\quad
m\to\infty\quad\Longrightarrow\quad
\eps_m\to 0^+\quad\Longrightarrow\quad
\ell({\tt T})=\lim\limits_{n\to\infty}\ell(v_n)=\ttt
$$
as required.
\end{proof}

\newpage{\large\section{Continued fractions}}

\subsection{Background on continued fractions}
\label{sec:con-frac-notation}

A \emph{$($simple$)$ continued fraction} is an
expression of the form
\be\label{eq:contd-frac1}
a_0+\frac{1}{a_1+\frac{1}{a_2+\frac{1}{a_3+\cdots}}},
\ee
where the \emph{coefficients} $a_n$ are chosen independently
of one another.
In this paper, we consider only continued fractions
with $a_0\in\Z$ and $a_n\in\N_1\defeq\N\setminus\{0\}$ for each $n\ge 1$;
note that \eqref{eq:contd-frac1} always
converges under these conditions. When the number of terms is finite,
one writes
\be\label{eq:contd-frac2}
[a_0;a_1,a_2,\ldots,a_n]\defeq
a_0+\frac{1}{a_1+\frac{1}{a_2+\cdots\frac{1}{\cdots+\frac{1}{a_n}}}},
\ee
and then $[a_0;a_1,a_2,\ldots,a_n]\in\Q$.
On the other hand, given an infinite sequence 
$(a_0,a_1,a_2,\ldots)\in\Z\times\N_1^\omega$, we denote by
$\alpha=[a_0;a_1,a_2,\ldots]$ the value of the infinite continued fraction
\eqref{eq:contd-frac1}; in this case, $\alpha$ is necessarily irrational,
i.e., $\alpha\in\RQ$.
It is known that the map
$$
\alpha=[a_0;a_1,a_2,\ldots]\mapsto(a_0,a_1,a_2,\ldots)
$$
is homeomorphism of $\RQ$ (with the usual topology inherited from $\R$)
onto $\Z\times\N_1^\omega$ (with the product topology).
We give a proof of this fact below; see 
Proposition~\ref{prop:homeomorphism}.

For any $\alpha=[a_0;a_1,a_2,\ldots]$,
the associated sequence $(a_0,a_1,a_2,\ldots)$
is used to derive the \emph{sequence of convergents} 
$(\frac{h_0}{k_0},\frac{h_1}{k_1},
\frac{h_2}{k_2},\ldots)\in\Q^\omega$, where the numerators and 
denominators are defined inductively by
\begin{alignat}{3}
\label{eq:hn-recursion}
h_0&\defeq a_0,&\qquad h_1&\defeq a_1a_0+1,
&\qquad\forall\,n\in\N:\quad h_{n+2}&\defeq a_{n+2}h_{n+1}+h_n,\\
\label{eq:kn-recursion}
k_0&\defeq 1,&\qquad k_1&\defeq a_1,
&\qquad\forall\,n\in\N:\quad k_{n+2}&\defeq a_{n+2}k_{n+1}+k_n.
\end{alignat}
By induction, $k_{n+2}>k_{n+1}$
for all $n\in\N$; in particular, it is clear
that $k_n\to\infty$ as $n\to\infty$. Using Khinchin~\cite[\S2]{Khin2}
we infer that
\be\label{eq:badabean}
\forall\,n\in\N:\quad\frac{h_n}{k_n}=[a_0;a_1,\ldots,a_n];
\ee
in other words, each convergent is a finite continued fraction
\eqref{eq:contd-frac2}  obtained by truncating the
infinite continued fraction \eqref{eq:contd-frac1}.
The following bounds are known
(see \cite[Thms.\ 9 and 13]{Khin2}):
\be\label{eq:ineqs}
\frac{1}{k_n(k_n+k_{n+1})}<\bigg|\alpha-\frac{h_n}{k_n}\bigg|
<\frac{1}{k_nk_{n+1}};
\ee
in particular, from \eqref{eq:ineqs} we deduce that
\be\label{eq:limit-convs}
\lim_{n\to\infty}\frac{h_n}{k_n}=\alpha.
\ee

The next result is also well known;  see \cite[Thm.\ 19]{Khin2}.

\newpage

\begin{lemma}\label{lem:too-close}
If $\alpha\in\RQ$, and the inequality
$$
\bigg|\alpha-\frac{p}{q}\bigg|<\frac{1}{2q^2}
$$
holds with some integers $p$ and $q>0$, then
$p/q$ is one of the convergents of $\alpha$. In particular,
the inequality $q\nearint{q\alpha}<\tfrac12$ implies that
$q=k_n$ for some $n\in\N$.
\end{lemma}

\subsection{A characterization of type}
\label{sec:type-characterization}

The next lemma plays a crucial role in this paper,
for it shows that the type $\tau(\alpha)$ 
of an irrational number $\alpha$ is determined
by the sequence of denominators $(k_n)$
of the convergents of $\alpha$.
The origin of this result is unclear, however, the statement and
sketch of the proof appear in a 2004 preprint of Sondow;
see Sondow~\cite[Thm.\ 1]{Sondow}. For the convenience of the reader,
we include a full proof here.

\begin{lemma}\label{eq:norah-jones}
In the notation of \S\ref{sec:con-frac-notation},
for every $\alpha\in\RQ$ we have
\be\label{eq:type-character}
\tau(\alpha)=\mathop{\overline{\rm lim}}\limits_{n\to\infty}
\frac{\log k_{n+1}}{\log k_n}.
\ee
\end{lemma}

\begin{proof}

By \eqref{eq:tau-defn} and the Dirichlet approximation theorem,
we have $\tau(\alpha)=\sup\cT(\alpha)$, where $\cT(\alpha)$ is the set
of real numbers defined by
$$
\cT(\alpha)\defeq\big\{\theta\ge 1:
\mathop{\underline{\rm lim}}\limits_{q\in\N}
~q^\theta\nearint{q\alpha}=0\big\}.
$$
We define
\be\label{eq:Theta-defn}
\forall\,n\ge 2:\quad\theta_n\defeq\frac{\log k_{n+1}}{\log k_n},
\mand
\Theta\defeq\mathop{\overline{\rm lim}}\limits_{n\to\infty}\theta_n.
\ee
To prove the lemma, we show that $\Theta=\sup\cT(\alpha)$.

For any $n\ge 2$, we have $k_{n+1}=k_n^{\theta_n}$,
hence from \eqref{eq:ineqs} it follows that
$$
\tfrac12k_n^{-\theta_n}<\big|k_n\alpha-h_n\big|<k_n^{-\theta_n}.
$$
Note that $|k_n\alpha-h_n|<\tfrac12$
since $k_n\ge 2$ and $\theta_n>1$, hence
$|k_n\alpha-h_n|=\nearint{k_n\alpha}$; this shows that
\be\label{eq:kn-nest}
\forall\,n\ge 2:\quad\tfrac{1}{2}<k_n^{\theta_n}\nearint{k_n\alpha}<1.
\ee

By \eqref{eq:Theta-defn}, for any $\eps>0$ there are infinitely
many $n\in\N$ such that $\theta_n>\Theta-\eps$. For any such
$n$, using the upper bound in \eqref{eq:kn-nest} we have
$$
k_n^{\Theta-2\eps}\nearint{k_n\alpha}
<k_n^{\theta_n-\eps}\nearint{k_n\alpha}<k_n^{-\eps}.
$$
Consequently,
$$
\mathop{\underline{\rm lim}}\limits_{q\in\N}
~q^{\Theta-2\eps}\nearint{q\alpha}=0,
$$
and so $\Theta-2\eps\in\cT(\alpha)$. Since $\eps>0$ is arbitrary,
we get that $\Theta\le\sup\cT(\alpha)$.

To finish the proof, we need to show that $\sup\cT(\alpha)\le\Theta$. 
Since $\Theta\ge 1$ (because $\theta_n>1$ for all~$n$),
there is nothing more to do in the case that $\sup\cT(\alpha)=1$.

From now on, suppose that $\sup\cT(\alpha)>1$. Let
$\theta\in\cT(\alpha)$ with $\theta>1$. Since
$$
\mathop{\underline{\rm lim}}\limits_{q\in\N}
~q^{\theta}\nearint{q\alpha}=0,
$$
we have $q^\theta\nearint{q\alpha}<\tfrac12$ for infinitely
many $q\in\N$. By Lemma~\ref{lem:too-close}, any such~$q$ has the form
$q=k_n$ for some $n$, and thus
$$
k_n^\theta\nearint{k_n\alpha}<\tfrac12
$$
for infinitely many $n$. Now \eqref{eq:kn-nest} implies that
$\theta<\theta_n$ for infinitely many $n$. 
By~\eqref{eq:Theta-defn}, one has $\theta_n<\Theta+\eps$ for
all sufficiently large $n$; therefore, $\theta<\Theta+\eps$. Since
this is true for any $\theta\in\cT(\alpha)$ with $\theta>1$, 
it follows that $\sup\cT(\alpha)\le\Theta+\eps$. Finally, as $\eps>0$
is arbitrary, we get that $\sup\cT(\alpha)\le\Theta$ as required.
\end{proof}

\subsection{Preservation of type under the action of ${\rm GL}_2(\Q)$}
\label{sec:GL2Q}

Let
\be\label{eq:tau-defn2}
\cT(\alpha)\defeq\big\{\theta\ge 1:
\mathop{\underline{\rm lim}}\limits_{q\in\N}
~q^\theta\nearint{q\alpha}=0\big\}
\ee
as in the proof of Lemma~\ref{eq:norah-jones}. Then
$\tau(\alpha)=\sup\cT(\alpha)$ for each $\alpha\in\RQ$.
The following lemma gives some invariance properties of the set $\cT(\alpha)$.

\begin{lemma}\label{lem:elem-opns}
For any irrational number $\alpha$, we have
\begin{itemize}
\item[$(i)$] $\cT(-\alpha)=\cT(\alpha)$;
\item[$(ii)$] $\cT(r\alpha)=\cT(\alpha)$ for all
$r\in\Q_+^\times\defeq\Q\cap(0,\infty)$;
\item[$(iii)$] $\cT(\alpha+1)=\cT(\alpha)$;
\item[$(iv)$] $\cT(\alpha^{-1})=\cT(\alpha)$.
\end{itemize}
\end{lemma}

\begin{proof}
Properties $(i)$ and $(iii)$ follow simply from
\eqref{eq:tau-defn2} since $\nearint{q\alpha}=\nearint{-q\alpha}$
and $\nearint{q\alpha}=\nearint{q(\alpha+1)}$
for all $q\in\Z$.

To prove property~$(ii)$, it suffices to show that
\be\label{eq:oneway}
\forall\,\alpha\in\RQ,~~\forall\,r\in\Q_+^\times:\quad
\cT(\alpha)\subset\cT(r\alpha),
\ee
since \eqref{eq:oneway} implies the opposite inclusion
$$
\forall\,\alpha\in\RQ,~~\forall\,r\in\Q_+^\times:\quad
\cT(r\alpha)\subset\cT(r^{-1}\cdot r\alpha)=\cT(\alpha).
$$
To prove \eqref{eq:oneway}, let $\theta\ge 1$ be an arbitrary
element of $\cT(\alpha)$, and write $r=u/v$ with
some integers $u,v\in\N_1$. For any given $\eps>0$, let
$\delta\in(0,\frac12)$ be such that $uv^\theta\delta<\eps$.
Since $\theta\in\cT(\alpha)$, the inequality
$q^\theta\nearint{q\alpha}<\delta$ holds for infinitely many $q\in\N_1$.
For any such $q$, there is an integer $p$ (the one closest to $q\alpha$)
for which
$$
q^\theta|q\alpha-p|<\delta.
$$
Multiplying both sides of the inequality by $rv^{\theta+1}=uv^\theta$, we get
$$
(vq)^\theta|vqr\alpha-up|<uv^\theta\delta
\quad\Longrightarrow\quad
(vq)^\theta\nearint{vq(r\alpha)}<\eps.
$$
Since this holds for infinitely many $q\in\N_1$,
it is clear that $\theta$ belongs to $\cT(r\alpha)$, and $(ii)$ is proved.

To prove property~$(iv)$, it is enough to show
\be\label{eq:oneway2}
\forall\,\alpha\in\RQ:\quad
\cT(\alpha)\subset\cT(\alpha^{-1}).
\ee
Fix $\alpha\in\RQ$. By $(i)$, we can assume $\alpha>0$.
Let $(\frac{h_0}{k_0},\frac{h_1}{k_1},\frac{h_2}{k_2},\ldots)$ be the
sequence of convergents to $\alpha$, and note that every $h_n$ and $k_n$
is positive. Finally, let $\theta\ge 1$ be an arbitrary
element of $\cT(\alpha)$. For any given $\eps>0$, let
$\delta\in(0,\frac12)$ be such that
$$
\forall\,n\in\N:\quad
\Big(\frac{h_n}{k_n}\Big)^\theta
\frac{\delta}{\alpha}<\eps;
$$
the existence of $\delta$ is a consequence of \eqref{eq:limit-convs}.
Since $\theta\in\cT(\alpha)$, we have
$q^\theta\nearint{q\alpha}<\delta$ for infinitely many $q\in\N_1$.
For any such $q$, Lemma~\ref{lem:too-close} shows that $q=k_n$
for some $n\in\N$.
Moreover, as in the proof of Lemma~\ref{eq:norah-jones}, 
we have $|k_n\alpha-h_n|=\nearint{k_n\alpha}$ for all large $n$. Therefore,
$$
k_n^\theta|k_n\alpha-h_n|\le k_n^\theta\nearint{k_n\alpha}<\delta,
$$
from which we deduce that
$$
h_n^\theta\nearint{h_n\alpha^{-1}}
\le h_n^\theta|h_n\alpha^{-1}-k_n|
= \frac{h_n^\theta}{\alpha k_n^\theta}
\cdot k_n^\theta|k_n\alpha-h_n|
<\Big(\frac{h_n}{k_n}\Big)^\theta
\frac{\delta}{\alpha}<\eps.
$$
Since $\eps$ is arbitrary, and this holds for infinitely many $n$,
it follows that $\theta$ belongs to $\cT(\alpha^{-1})$, completing
the proof of $(iv)$.
\end{proof}

\bigskip

\begin{proof}[Proof of Theorem~\ref{thm:main-GL2}]
Let $G\defeq GL_2(\Q)$, and recall the action of $G$ on $\RQ$:
$$
g \alpha\defeq\frac{a\alpha+b}{c\alpha+d}
\quad\text{for any~}g=\begin{pmatrix}a&b\\c&d\end{pmatrix}\in G
\text{~and~}\alpha\in\RQ.
$$
We need to show that $\tau(g\alpha)=\tau(\alpha)$ for
all $\alpha\in\RQ$ and $g\in G$.

Let $H$ be the subgroup of $G$ consisting
of matrices $h$ such that $\cT(h\alpha)=\cT(\alpha)$ \emph{for all}
$\alpha\in\RQ$. To prove the theorem, we must show that $H=G$.
Note that the properties $(i)$, $(ii)$, $(iii)$, and $(iv)$ 
in Lemma~\ref{lem:elem-opns} imply that $H$ contains all of the matrices
$$
\begin{pmatrix}-1&\\&1\end{pmatrix},\quad
\begin{pmatrix}r&\\&1\end{pmatrix}\quad
\text{for all~}r\in\Q_+^\times,\quad
\begin{pmatrix}1&1\\&1\end{pmatrix},
\quad\text{and}\quad
w\defeq\begin{pmatrix}&1\\1&\end{pmatrix}.
$$

Under the action of $G$ on $\RQ$ by M\"obius transformations,
the center
$$
Z\defeq\left\{\begin{pmatrix}z&\\&z\end{pmatrix}:
z\in\Q\setminus\{0\}\right\}
$$
acts trivially; hence $Z\subset H$. The
set $T$ of diagonal matrices is also contained in~$H$,
since any diagonal matrix can be decomposed as
$$
\begin{pmatrix}a&\\&b\end{pmatrix}
=\begin{pmatrix}\pm 1&\\&1\end{pmatrix}
\begin{pmatrix}|a/b|&\\&1\end{pmatrix}
\begin{pmatrix}b&\\&b\end{pmatrix}
$$
for some choice of sign. Next, taking into account that
$$
\forall\,x\in\Q\setminus\{0\}:\quad
\begin{pmatrix}1&x\\&1\end{pmatrix}
=\begin{pmatrix}x&\\&1\end{pmatrix}
\begin{pmatrix}1&1\\&1\end{pmatrix}
\begin{pmatrix}x^{-1}&\\&1\end{pmatrix},
$$
we see that the collection of unipotent matrices
$$
N\defeq\left\{\begin{pmatrix}1&x\\&1\end{pmatrix}:x\in\Q\right\}
$$
is contained in $H$. 
To summarize, we have seen that $w\in H$, $T\subset H$, and $N\subset H$.
By the Bruhat decomposition (see, for example, Borel~\cite[\S14.12]{Borel}),
we have $G=TN\cup TNwN$,
and so we conclude that $H=G$.
\end{proof}

\bigskip

\begin{proof}[First proof of Theorem~\ref{thm:densely-onto}]
The map $\tau:\RQ\to[1,\infty]$ is surjective. Indeed,
for any given $\ttt\in[1,\infty]$, using
I changed this because now we have $x<y$ when $(x,y)$ lies in $V$
the initial vertex $v\defeq(1,2)$ in Proposition~\ref{prop:happy-trails}
and its proof, we construct a sequence $a=(a_n)_{n\in\N}$
of positive integers and a trail
$$
{\tt T}={\tt Tr}(v,a):\quad
v\squig{a_0}
v_0\squig{a_1} v_1
\squig{a_2} v_2
\squig{a_3} v_3
\squig{a_4} \cdots
$$
that has an $\ell$-limit equal to $\ttt$. Then
$\alpha\defeq[a_0;a_1,a_2,\ldots]$ is an irrational number
for which $\tau(\alpha)=\ell({\tt T})=\ttt$. This proves the
surjectivity of the map $\tau$.
Moreover, for any $x\in\Q$ we have
$\tau(\alpha+x)=\tau(\alpha)=\ttt$ by Theorem~\ref{thm:main-GL2}, since
$$
\alpha+x=g\alpha\qquad\text{with}\quad
g\defeq\begin{pmatrix}1&x\\&1\end{pmatrix}\in G.
$$
Therefore, $\alpha+\Q\subset\tau^{-1}(\{\ttt\})$, which implies that $\tau^{-1}(\{\ttt\})$
is dense in $\RQ$.
\end{proof}

\subsection{Trails of an irrational number}
\label{sec:fraction-trails}

For any $\alpha\in\RQ$, let the sequences $(a_n)_{n\in\N}$ and
$(k_n)_{n\in\N}$ be defined as in \S\ref{sec:con-frac-notation}. Defining
$$
\forall\,n\in\N_1:\quad
v_n\defeq(k_n,k_{n+1}),
$$
we obtain a sequence of vertices in the oriented graph $\cG$ 
described in \S\ref{sec:edges-and-trails}; in other words,
$(v_n)_{n\in\N_1}\subset V$. Moreover, by \eqref{eq:squig-games}
and the recursive definition \eqref{eq:kn-recursion},
each pair of consecutive vertices is connected by a
directed edge:
$$
\forall\,n\in\N_1:\quad v_n\squig{a_{n+2}} v_{n+1}.
$$
In this way, we obtain an infinite trail
$$
{\tt T}_\alpha\defeq{\tt Tr}\big(v_1,(a_n)_{n\ge 3}\big):\quad
v_1
\squig{a_3} v_2
\squig{a_4} v_3
\squig{a_5}  v_4 \cdots
$$
of the type described in \S\ref{sec:edges-and-trails}. Moreover,
using Lemma~\ref{eq:norah-jones} we have
$$
\overline\ell({\tt T}_\alpha)
=\mathop{\overline{\rm lim}}\limits_{n\to\infty}\ell(v_n)
=\mathop{\overline{\rm lim}}\limits_{n\to\infty}
\frac{\log k_{n+1}}{\log k_n}=\tau(\alpha).
$$
To sum up, every irrational number $\alpha$ determines an infinite
trail ${\tt T}_\alpha$ with an initial vertex $v_1=(a_1,a_2a_1+1)$
and an $\ell$-limsup $\overline\ell({\tt T}_\alpha)=\tau(\alpha)$.

Conversely, for any trail ${\tt T}_\alpha$ as above,
the sequences $(a_n)_{n\in\N}$ and $(k_n)_{n\in\N}$ are uniquely determined by
the vertices $(v_n)_{n\in\N_1}$, except for $a_0$. Thus, the irrational number
$\alpha$ is uniquely determined by its trail up to translation by an integer.

\subsection{Heads or tails?}
\label{sec:heads-or-tails}

Let $\alpha=[a_0;a_1,a_2,\ldots]\in\RQ$ be given.
For every $n\in\N$, we define the \emph{$n$-th head of $\alpha$}
to be the rational number 
$$
\alpha_{\le n}\defeq[a_0;a_1,a_2,\ldots,a_n],
$$
which is just the $n$-th convergent $h_n/k_n$ to $\alpha$;
see \eqref{eq:badabean}. Because $\alpha$ is the limit of its convergents
(see \eqref{eq:limit-convs}), the various heads of an
irrational number determine its location in $\R$.
We also have the following result.

\begin{lemma}\label{lem:heading-out}
For any $\eps>0$, there exists $n_0=n_0(\eps)>0$ such that
$$
\forall\,\alpha,\beta\in\RQ,~~\forall\,n\ge n_0:\quad
\alpha_{\le n}=\beta_{\le n}
\quad\Longrightarrow\quad
|\alpha-\beta|<\eps.
$$
Also, for any $\alpha\in\RQ$ and $n\in\N$, $n\ge 2$,
there exists $\delta=\delta(\alpha,n)>0$ such that
$$
\forall\,\beta\in\RQ:\quad
|\alpha-\beta|<\delta
\quad\Longrightarrow\quad
\alpha_{\le n}=\beta_{\le n}.
$$
\end{lemma}

\begin{proof}
For any $\alpha,\beta\in\RQ$, if $\alpha_{\le n}=h_n/k_n=\beta_{\le n}$,
then by the triangle inequality and \eqref{eq:ineqs} one has
$$
|\alpha-\beta|\le\bigg|\alpha-\frac{h_n}{k_n}\bigg|
+\bigg|\beta-\frac{h_n}{k_n}\bigg|<\frac{2}{k_nk_{n+1}}.
$$
On the other hand, by \eqref{eq:kn-recursion}, the lower bound $k_n\ge n-1$
holds for all $n\in\N$ regardless of the values of $\alpha$ and $\beta$.
This implies the first statement of the lemma.

Our proof of the second statement uses ideas from the proof
of \cite[Thm.~18]{Khin2}. Given $\alpha\in\RQ$ and $n\in\N$, $n\ge 2$,
the number $\alpha$ lies strictly between the consecutive convergents
$\alpha_{\le n}=h_n/k_n$ and $\alpha_{\le n+1}=h_{n+1}/k_{n+1}$.
Let $\delta>0$ be small enough so that
$$
\delta<\min\bigg\{\bigg|\alpha-\frac{h_n}{k_n}\bigg|,
\bigg|\alpha-\frac{h_{n+1}}{k_{n+1}}\bigg|\bigg\}.
$$
For any $\beta\in\RQ$ for which $|\alpha-\beta|<\delta$, the number $\alpha$
lies closer to $\beta$ than it does to either of its convergents 
$h_n/k_n$ or $h_{n+1}/k_{n+1}$; therefore $\beta$ lies strictly between
$h_n/k_n$ and $h_{n+1}/k_{n+1}$. Consequently,
$$
\bigg|\beta-\frac{h_n}{k_n}\bigg|+
\bigg|\beta-\frac{h_{n+1}}{k_{n+1}}\bigg|
=\bigg|\frac{h_n}{k_n}-\frac{h_{n+1}}{k_{n+1}}\bigg|
=\frac{1}{k_nk_{n+1}}<\frac{1}{2k_n^2}+\frac{1}{2k_{n+1}^2},
$$
where we used the corollary to \cite[Thm.~2]{Khin2} to derive second equality,
and we used the arithmetic-geometric mean inequality in the last step.
Hence, either
$$
\bigg|\beta-\frac{h_n}{k_n}\bigg|<\frac{1}{2k_n^2}
\qquad\text{or}\qquad
\bigg|\beta-\frac{h_{n+1}}{k_{n+1}}\bigg|<\frac{1}{2k_{n+1}^2}.
$$
Lemma~\ref{lem:too-close} now shows that $h_n/k_n$ or $h_{n+1}/k_{n+1}$
is one of the convergents of $\beta$, and in either case we have
$\alpha_{\le n}=\beta_{\le n}$.
\end{proof}

\begin{proposition}\label{prop:homeomorphism}
The map $\phi:\RQ\to\Z\times\N_1^\omega$ given by
$$
\alpha=[a_0;a_1,a_2,\ldots]\mapsto\phi(\alpha)\defeq(a_n)_{n\in\N}
$$
is a homeomorphism. The inverse of $\phi$ is the map
$\psi:\Z\times\N_1^\omega\to\RQ$ given by
$$
a\mapsto\psi(a)\defeq\lim_{n\to\infty}\frac{h_n}{k_n},
$$
where for any given sequence $a=(a_n)$
one uses \eqref{eq:hn-recursion} and \eqref{eq:kn-recursion}
to define the corresponding sequence of convergents $(h_n/k_n)_{n\in\N}$.
\end{proposition}

\begin{proof}
Since
$$
\alpha_{\le n}=[a_0;a_1,a_2,\ldots,a_n]=\frac{h_n}{k_n},
$$
for all $n\in\N$, one sees that $\phi$ and $\psi$ are inverse maps.

To show that $\phi$ is continuous, it is enough to show that $\phi^{-1}(\cA)$
is open for every basic open set of the form
\be\label{eq:basicAdefn}
\cA\defeq\{(a_0,a_1,\ldots,a_n)\}\times\N_1^\omega,
\ee
where $n\ge 2$ and $(a_0,a_1,\ldots,a_n)\in\Z\times\N_1^n$.
Writing $H\defeq[a_0;a_1,\ldots,a_n]$, the set $\phi^{-1}(\cA)$ consists
of those irrational numbers $\alpha$ for which $\alpha_{\le n}=H$. 
For any fixed $\alpha\in\phi^{-1}(\cA)$, let $\delta=\delta(\alpha,n)>0$ have the
property stated in Lemma~\ref{lem:heading-out}. The set
$$
\cB\defeq\big\{\beta\in\RQ:|\alpha-\beta|<\delta\big\}
$$
is open in $\RQ$ and contains $\alpha$, and by Lemma~\ref{lem:heading-out},
$\beta_{\le n}=\alpha_{\le n}$ for every $\beta\in\cB$, i.e.,
$\beta_{\le n}=H$ for every $\beta\in\cB$. This shows that
$\cB\subset\phi^{-1}(\cA)$, and 
we deduce that $\phi^{-1}(\cA)$ is open.

Next, we show that $\psi$ is continuous. For this, it suffices to show that
$\psi^{-1}(\cB)$ is open for every basic open set of the form
$$
\cB\defeq (\RQ)\cap\cI,
$$
where $\cI\subset\R$ is an open interval. For any fixed $a\in\psi^{-1}(\cB)$,
put $\alpha\defeq\psi(a)$, which belongs to $\cB$. Let $\eps>0$ be small enough
so that every $\beta\in\RQ$ satisfying $|\alpha-\beta|<\eps$ in contained in $\cB$.
Let $n_0=n_0(\eps)>0$ have the property stated in Lemma~\ref{lem:heading-out}.
Finally, let $n\ge n_0$, put $H\defeq \alpha_{\le n}=[a_0;a_1,\ldots,a_n]$, and
define $\cA$ as in  \eqref{eq:basicAdefn}. Then for every $a\in\cA$ we have
$$
\psi(a)_{\le n}=H=\alpha_{\le n}
\quad\Longrightarrow\quad
\big|\alpha-\psi(a)\big|<\eps
\quad\Longrightarrow\quad
\psi(a)\in\cB,
$$
and thus $\cA\subset\psi^{-1}(\cB)$. We conclude that $\psi^{-1}(\cB)$ is open.
\end{proof}

\bigskip

Next, for every $n\in\N$, let $\alpha_{\ge n}\in\RQ$ be the
\emph{$n$-th tail of $\alpha$}:
$$
\alpha_{\ge n}\defeq[a_n;a_{n+1},a_{n+2},\ldots],
$$
which is also an irrational number. Our next result shows that the type of $\alpha$
is determined by any one of its tails.

\begin{theorem}\label{thm:tails}
We have $\tau(\alpha_{\ge n})=\tau(\alpha)$ for all $\alpha\in\RQ$ and $n\in\N$.
\end{theorem}

\begin{proof}
For every $\alpha\in\RQ$ we have
$$
\alpha_{\ge 0}=\alpha
\mand
\forall\,n\in\N:\quad
\alpha_{\ge n+1}=(\alpha_{\ge n})_{\ge 1}.
$$
Hence, if we demonstrate that
\be\label{eq:neil-young}
\forall\,\alpha\in\RQ:\quad
\tau(\alpha_{\ge 1})=\tau(\alpha),
\ee
then the general result follows via an inductive argument.

To prove \eqref{eq:neil-young}, observe that
$$
\alpha=a_0+\frac{1}{a_1+\frac{1}{a_2+\cdots}}
=a_0+\frac{1}{\alpha_{\ge 1}},
$$
and therefore (as in \eqref{eq:Mobius}) we have
$$
\alpha_{\ge 1}=\frac{1}{\alpha-a_0}=g\alpha
\qquad\text{with}\quad
g\defeq\begin{pmatrix}&1\\1&-a_0\end{pmatrix}\in GL_2(\Q).
$$
Now the relation \eqref{eq:neil-young} follows immediately
from Theorem~\ref{thm:main-GL2}.
\end{proof}

\begin{corollary}\label{cor:get-off-my-tail}
If $\alpha,\beta\in\RQ$ and $\alpha_{\ge m}=\beta_{\ge n}$
for some $m,n\in\N$, then $\alpha$~and~$\beta$ have the same type.
\end{corollary}

\begin{remark*}
An old result of Serret
$($see Perron~{\rm\cite[Satz 23]{Perron}}$)$
asserts that two irrational numbers
$\alpha$ and $\beta$ are equivalent under the action
of $GL_2(\Z)$ $($by M\"obius transformations$)$
if and only if $\alpha_{\ge m}=\beta_{\ge n}$
for some $m,n\in\N$. Using this result, one can simplify
some of the arguments above. For example, since
$\alpha_{\ge n}$ and $\alpha$ are equivalent under
$GL_2(\Z)$, Theorem~\ref{thm:tails} can be proved by
combining Serret's result with Theorem~\ref{thm:main-GL2},
along with the observation that the scalar matrices in
$GL_2(\Q)$ act trivially.
\end{remark*}

\begin{remark*}
In descriptive set theory, for every set $\Omega$ there is
an important relation $E_t(\Omega)$ known as the
\emph{tail equivalence relation}
$($see, e.g., Dougherty et al~{\rm\cite[\S2]{DJK}}$)$,
which is defined on the collection of sequences $\Omega^\omega$ by
$$
xE_t(\Omega)y\quad\Longleftrightarrow\quad
\exists\,m,n\in\N,~\forall\,k\in\N:\quad x(m+k)=y(n+k).
$$ 
In terms of this relation, Proposition \ref{prop:homeomorphism} and
Corollary \ref{cor:get-off-my-tail} together show that if $a_0,b_0\in\Z$
and $(a_{n+1})_{n\in\N}E_t(\N_1)(b_{n+1})_{n\in\N}$, then
$\psi((a_n)_{n\in\N})$ and $\psi((b_n)_{n\in\N})$ have the same type.
\end{remark*}

\bigskip\begin{proof}[Second proof of Theorem~\ref{thm:densely-onto}]
We have already seen that the map $\tau$ is surjective.
Given $\ttt\in[1,\infty]$, let $\alpha\in\RQ$ be such that $\tau(\alpha)=\ttt$.

Now, let $\beta\in\RQ$ be arbitrary. For every $n\in\N$, let
$\gamma_n$ be the irrational number whose head is $\beta_{\le n}$
and whose tail is $\alpha_{\ge n+1}$, i.e.,
$$
\gamma_n\defeq[b_0;b_1,\ldots,b_n,a_{n+1},a_{n+1},\ldots].
$$
Since $(\gamma_n)_{\le n}=\beta_{\le n}$
for each $n$, Lemma~\ref{lem:heading-out} implies that
$\gamma_n\to\beta$ as $n\to\infty$. On the other hand,
since $(\gamma_n)_{\ge n+1}=\alpha_{\ge n+1}$ for each $n$,
Corollary~\ref{cor:get-off-my-tail} shows that
$\tau(\gamma_n)=\tau(\alpha)=\ttt$; thus, every number
$\gamma_n$ lies in $\tau^{-1}(\{\ttt\})$.
\end{proof}

\newpage{\large\section{Examples of complete sets}
\label{sec:various-complete-sets}}

Our proof of Theorem~\ref{thm:main2} (see \S\ref{sec:complexity} below)
ultimately relies on the fact that
$$
\cN_\infty\defeq\big\{\beta\in\N_1^\omega:\lim\limits_{n\to\infty}\beta(n)=\infty\big\}
$$
is a complete $\borelp{0}{3}$, a result that is well-known in descriptive
set theory. For the convenience of the reader, we include a proof here,
which borrows extensively from the material presented in Kechris~\cite[\S23.A]{K}.

\begin{lemma}\label{lem:Cantor-pi02}
Let $2^\omega$ be the Cantor space of infinite binary sequences.
The set
$$
\cL\defeq\{f\in 2^\omega:f(n)=0
\text{~for all but finitely many $n\in\N$}\}.
$$
is a complete $\borels{0}{2}$.
\end{lemma}

\begin{proof}
For each $n\in\N$, the set $U_n\defeq\{f\in 2^\omega:f(n)=0\}$
is clopen (closed and open), and therefore
$$
\cL=\bigcup\limits_{n_0\in\N}\bigcap\limits_{n\ge n_0}U_n
$$
is a $\borels{0}{2}$ subset of the Polish space $2^\omega$.
On the other hand, as both $\cL$ and $2^\omega\setminus\cL$
are dense in $2^\omega$, the set $\cL$ cannot be a $G_\delta$ by
Baire's theorem (see \cite[Thm.~8.4]{K}); in other words,
$\cL\not\in\borelp{0}{2}$. Now
Wadge's theorem (see \cite[Thm.~22.10]{K}) shows that $\cL$ is a
complete $\borels{0}{2}$.
\end{proof}

\begin{lemma}\label{lem:matrix-pi03}
Let $2^{\omega\times\omega}$ be the space of infinite binary matrices,
and let $\cM$ be the subset  consisting of
infinite binary matrices whose columns each have at most finitely many $1$'s.
Then $\cM$ is a complete $\borelp{0}{3}$.
\end{lemma}

\begin{proof}
For fixed $m,n\in\N$, the set
$U_{m,n}\defeq\{f\in 2^{\omega\times\omega}:f(m,n)=0\}$
is clopen, and thus
$$
\cM=\bigcap\limits_{m\in\N}\bigcup\limits_{n_0\in\N}
\bigcap\limits_{n\ge n_0}U_{m,n}
$$
is a $\borelp{0}{3}$ subset of the Polish space 
$2^{\omega\times\omega}$.

Now let $\cX$ be an arbitrary zero-dimensional Polish space. For any
$\cA\in\borelp{0}{3}(\cX)$,
there is a family $(\cA_m)_{m\in\N}$ of $\borels{0}{2}$ 
subsets of $\cX$ such that $\cA=\bigcap_{m\in\N}\cA_m$.
Since each $\cA_m\le_W\cL$ by Lemma~\ref{lem:Cantor-pi02},
there is a continuous map $g_m:\cX\to 2^\omega$ for which
$\cA_m=g_m^{-1}(\cL)$. We define $g:\cX\to 2^{\omega\times\omega}$ by 
$$
\forall\,x\in\cX,~~\forall\,m,n\in\N:\quad g(x)(m,n)\defeq g_m(x)(n).
$$ 
Then $g$ is continuous, and
$$
x\in \cA=\bigcap_{m\in\N}g_m^{-1}(\cL)
\quad\Longleftrightarrow\quad\forall\,m\in\N:\quad g_m(x)\in\cL
\quad\Longleftrightarrow\quad g(x)\in \cM.
$$
Thus we have $\cA\le_W\cM$ as required.
\end{proof}

\begin{lemma}\label{lem:diverges-pi03}
The set $\cN_\infty\defeq\big\{\beta\in\N_1^\omega:\lim\limits_{n\to\infty}\beta(n)=\infty\big\}$ is a complete $\borelp{0}{3}$.
\end{lemma}

\begin{proof}
For any $k\in\N$ the set $U_k\defeq\{\beta\in\N_1^\omega\mid\beta(n)\ge k\}$ 
is clopen, and therefore
$$
\cN_\infty=\bigcap\limits_{k\in\N}\bigcup\limits_{n_0\in\N}
\bigcap\limits_{n\ge n_0}U_k
$$
is a $\borelp{0}{3}$ subset of the Polish space $\N_1^\omega$.

Fix a bijection $\langle\cdot,\cdot\rangle:\N\times\N\to\N$
with the property that 
$\langle m,n\rangle\ge m$ for all $m,n\in\N$ (for instance,
one can take $\langle m,n\rangle\defeq 2^n(2m+1)-1$). We define a map 
$f:2^{\omega\times\omega}\to\N^\omega$ by 
$$
\forall\,\alpha\in 2^{\omega\times\omega},~~\forall\,k\in\N:\quad
f(\alpha)(k)\defeq\begin{cases}
k&\quad\hbox{if $k=\langle m,n\rangle$ and $\alpha(m,n)=0$},\\
m&\quad\hbox{if $k=\langle m,n\rangle$ and $\alpha(m,n)=1$}.\\
\end{cases}
$$ 
Then $f$ is continuous, and $\cM=f^{-1}(\cN_\infty)$.

Now let $\cX$ be an arbitrary zero-dimensional Polish space. For any
$\cA\in\borelp{0}{3}(\cX)$, Lemma~\ref{lem:matrix-pi03}
shows the existence of a continuous map
$g:\cX\to 2^{\omega\times\omega}$ such that
$\cA=g^{-1}(\cM)$. Then $f\circ g:\cX\to\N^\omega$ is continuous, and
$\cA=(f\circ g)^{-1}(\cN_\infty)$. Thus we have $\cA\le_W\cN_\infty$.
\end{proof}

\newpage{\large\section{Complete functions}
\label{sec:complete functions}}

The aim of this section is to prove Theorem~\ref{thm:main3}.

\subsection{Background on trees}
\label{sec:trees}
Let $A$ be a nonempty set. Following \cite[2.A]{K}, for each $n\in\N$ we denote
by $A^n$ the set of finite sequences $s=(s_0,\ldots,s_{n-1})$ of \emph{length}
${\tt len}(s)=n$. In particular, $A^0=\{\varnothing\}$, where $\varnothing$ is
the empty sequence. The set $A^\omega$ consists of infinite (countable) sequences
$s=(s_0,s_1\ldots)$ of length ${\tt len}(s)=\infty$. We put
$$
A^{<\omega}\defeq\bigcup_{n\in\N}A^n\mand A^{\le\omega}\defeq A^{<\omega}\cup A^\omega.
$$
Given $s\in A^{\le\omega}$ and $n\le{\tt len}(s)$
we denote $s|_n\defeq(s_0,\ldots,s_{n-1})$, where formally we have
$s|_0=\varnothing$, and $s|_\infty=s$ in the case that $s\in A^\omega$.
If $s\in A^{<\omega}$ and $t\in A^{\le\omega}$, we say that $s$ is an
\emph{initial segment} of $t$ and that $t$ is an extension of $s$ (and we
write $s\subseteq t$) if $s=t|_n$ with $n={\tt len}(s)$.

A \emph{tree} on $A$ is a subset $T\subset A^{<\omega}$ that is
\emph{closed under initial segments}, i.e., for any given $t\in T$,
every initial segment $s$ of $t$ is also contained in $T$.
The \emph{body} of a tree $T$ is the set
$$
[T]\defeq\big\{x\in A^\omega:x\vert_n\in T\text{~for all~}n\in\N\big\}.
$$ 
A tree $T$ is said to be \emph{pruned} if every element of $T$ has a proper
extension in $T$; equivalently, every element of $T$ is an initial segment
of some element of $[T]$.

For any sequence $s\in A^{<\omega}$ we denote
$$
N_s\defeq\big\{x\in A^\omega:s\subseteq x\big\}.
$$
Then $N_s$ is the basic clopen neighborhood of $A^\omega$ associated with $s$.
If $T$ is a tree on $A$ and $s\in T$, then we denote
$$
T_s\defeq\big\{t\in T:s\subseteq t\text{~or~}t\subseteq s\big\}.
$$

\subsection{A key lemma}
\label{sec:key-lime-pie}
 To establish the universality of the type function,
we use a generalization of the Lemma \ref{lem:targeting}, the crucial
``targeting lemma.'' Recall that
$$
V\defeq\big\{(x,y)\in\N_1^2:x<y\}.
$$
Let $\gB$ be a fixed subset of 
$$
W\defeq\big\{\big((x_1,y_1),(x_2,y_2)\big)\in V^2:x_2=y_1\big\},
$$
and put
$$
P=P_\gB\defeq\big\{(v_n)_{n\in\N_1}\in V^\omega:
(v_n,v_{n+1})\in \gB\text{~for all~}n\in\N_1\}.
$$
By \cite[Proposition 2.4]{K}, since $P$ is a closed subset of $V^\omega$,
there exists a pruned tree $T$ on $V$ with the property that $P=[T]$;
the tree is given by
$$
T=T_\gB\defeq\bigcup_{n\in\N}
\big\{ (v_1,\cdots ,v_n)\in V^n:(v_m,v_{m+1})\in \gB\text{~for~}1\le m<n\big\}.
$$
For considerations of density, for each $s\in T$ we use the notation $T_s$ defined
above, and we write $P_s\defeq P\cap N_s=[T_s]$.

\begin{definition*}
For fixed $\gB$ as above and $\gb\in\R\cup\{-\infty\}$,
we say that a map $\gell:V\to[\gb,\infty)$ is $(\gB,\gb)$-\emph{target-controlled}
if it has the following property. For any $\eps\in(0,1)$, there is a positive integer
$M=M_{\gB,\gb}(\eps)$ such that for every $\ttt\in[\gb,\infty)\cap\R$ and $v=(x,y)\in V$
with $y>M$, there exists a vertex $w\in V$ for which $(v,w)\in \gB$ and
$\gell(w)\in[\ttt,\ttt+\eps)$.
\end{definition*}

The prototype for this definition is the specific
function $\ell$ considered earlier in connection with
the type function $\tau$; see \eqref{eq:ell-defn} for the
definition.
In view of Lemma \ref{lem:targeting}, the map $\ell$
is $(B,1)$-target-controlled, where
$$
B\defeq\big\{\big((x_1,y_1),(x_2,y_2)\big)\in V^2:x_2=y_1
\text{~and~}y_2=ay_1+x_1\text{~for some~}a\in\N_1\big\}.
$$
Note that, in this case, every integer larger than
$(10/\eps)^{1/\eps}$ is an acceptable value of
$M_{B,1}(\eps)$ for any $\eps\in(0,1)$.

In the general case, where $\gell$ is an arbitrary
$(\gB,\gb)$-target-controlled map, we need the following key lemma.
For any $n\in\N_1$, let $\eps_n\defeq2^{-n}$, and let
$M_n$ be an acceptable value of $M_{\gB,\gb}(\eps_n)$.
We can assume that the sequence $(M_n)_{n\in\N_1}$
is strictly increasing.

\begin{lemma}\label{lem:tre}
Let $s\in T$ such that $s=\varnothing$ or ${\tt len}(s)>M_1$.
For any function $u:T_s\to [\gb,\infty)$, there exists
$v=(v_n)_{n\in\N_1}\in P_s$ such that 
$$
\mathop{\overline{\rm lim}}\limits_{m\in\N_1}u(v|_m)
=\mathop{\overline{\rm lim}}\limits_{n\in\N_1}\gell(v_n).
$$
\end{lemma}

\begin{proof}
By induction on $n$, we construct $\ttt_n\in [\gb,\infty)\cap\R$ and $v_n\in V$ as follows. 

Let $\ttt_1\in [\gb,\infty)\cap\R$ be arbitrary.
Put $l_1\defeq{\tt len}(s)$. If $s=\varnothing$, then $l_1=0$ and we
set $v_0\defeq (M_1,M_1+1)$. If $s\ne\varnothing$, then $l_1>M_1$,
and writing $s=(s_1,\ldots,s_{l_1})$ we put $v_m\defeq s_m$ for $1\le m\le l_1$.
In both cases, $v_{l_1}\defeq (x_{l_1},y_{l_1})\in V$ and we have
$y_{l_1}>M_1$ (since $s\in T$).

Since $\gell$ is $(\gB,\gb)$-target-controlled, there is a vertex
$v_{l_1+1}=(x_{l_1+1},y_{l_1+1})\in V$ such that
$(v_{l_1},v_{l_1+1})\in \gB$ and 
$\gell (v_{l_1+1})\in [\ttt_1,\ttt_1+\eps_1)$. Moreover,
$$
y_{l_1+1}>x_{l_1+1}=y_{l_1}>M_1.
$$
Iterating this process, for each $j\ge 1$ we obtain 
$v_{l_1+j}=(x_{l_1+j},y_{l_1+j})\in V$ such that
$(v_{l_1+j-1},v_{l_1+j})\in \gB$,
$\gell (v_{l_1+j})\in [\ttt_1,\ttt_1+\eps_1)$, and
$$
y_{l_1+j}>\cdots>y_{l_1+1}>y_{l_1}>M_1.
$$
Then, for some sufficiently large $j$, we also have $y_{l_1+j}>M_2$. Let
$l_2\defeq l_1+j$ for some such $j$, and put
$$
\ttt_2\defeq\begin{cases}
\displaystyle\max\limits_{l_1<m\le l_2}u\big((v_1,\ldots,v_m)\big)
&\quad\hbox{if this maximum lies in $\R$},\\
-2&\quad\hbox{otherwise}.
\end{cases}
$$
Since $\gell$ is $(\gB,\gb)$-target-controlled, there is a vertex
$v_{l_2+1}=(x_{l_2+1},y_{l_2+1})\in V$ such that
$(v_{l_2},v_{l_2+1})\in \gB$ and 
$\gell (v_{l_2+1})\in [\ttt_2,\ttt_2+\eps_2)$, and the above process can
be repeated.

Continuing in this manner, we obtain sequences
$l\defeq(l_n)_{n\in\N_1}$, $v\defeq(v_n)_{n\in\N_1}$, and $\ttt\defeq(\ttt_n)_{n\in\N_1}$ such that
\begin{itemize}
\item[$(i)$] $l$ is strictly increasing, with $y_{l_n}>M_n$ for each $n\in\N_1$;
\item[$(ii)$] $v_n\in V$ and
$(v_n,v_{n+1})\in \gB$ for all $n\in\N_1$, hence $v\in P$;
\item[$(iii)$] $\ttt$ is a real sequence, with
$\ttt_n$ for $n\ge 2$ defined by
$$
\ttt_n\defeq\begin{cases}
\displaystyle\max\limits_{l_{n-1}<m\le l_n}u(v|_m)
&\quad\hbox{if this maximum lies in $\R$},\\
-n&\quad\hbox{otherwise}.
\end{cases}
$$
\item[$(iv)$] $\gell(v_m)\in[\ttt_n,\ttt_n+\eps_n)$ for
$l_n<m\le l_{n+1}$.
\end{itemize}

\bigskip

For any integer $m>l_1$, by $(i)$ there is a unique 
$n\in\N_1$ with $l_n<m\le l_{n+1}$, and so by $(iii)$ and $(iv)$ we have
$$
u(v|_m)\le\ttt_{n+1}\le\gell(v_{l_{n+2}}),
$$
which implies that
$$
\mathop{\overline{\rm lim}}\limits_{m\in\N_1}u(v|_m)
\le\mathop{\overline{\rm lim}}\limits_{n\in\N_1}\gell(v_n).
$$

It remains to establish the inverse inequality
\be\label{eq:opp-ineq}
\mathop{\overline{\rm lim}}\limits_{n\in\N_1}\gell(v_n)
\le\mathop{\overline{\rm lim}}\limits_{k\in\N_1}u(v|_k).
\ee
For this, we can clearly assume that
$$
\mathop{\overline{\rm lim}}\limits_{n\in\N_1}\gell(v_n)
=\lim\limits_{j\in\N}\gell(v_{m_j})\ne -\infty
$$
with some strictly increasing sequence $(m_j)_{j\in\N}\subset\N_1$.
Fixing $\eps>0$, $n\in\N_1$ with $\eps_n<\eps$, and 
$j$ such that $l_n<m_j\le l_{n+1}$, we have
$$
\gell(v_{m_j})<\ttt_n+\eps_n
=\eps_n+\begin{cases}
\displaystyle\max\limits_{l_{n-1}<k\le l_n}u(v|_k)
&\quad\hbox{if this maximum lies in $\R$},\\
-n&\quad\hbox{otherwise}.
\end{cases}
$$ 
Thus $\gell (v_{m_j})<u(v|_{k_j})+\eps$ holds with some $k_j$ in the range
$l_{n-1}<k_j\le l_n$. This shows that
$$
\mathop{\overline{\rm lim}}\limits_{n\in\N_1}\gell(v_n)
\le\mathop{\overline{\rm lim}}\limits_{k\in\N_1}u(v|_k)+\eps,
$$
and as $\eps$ is arbitrary, we deduce \eqref{eq:opp-ineq}. This completes the proof.
\end{proof}

\subsection{Background on games}

Let $A$ be a nonempty set, $\cT$ be a nonempty pruned tree on $A$, and $\cX\subset [\cT]$. We consider the \emph{game} $G(\cT,\cX)$ played as follows:
\begin{alignat*}{50}
&\text{Player~1}:\quad&a_0\quad&&a_2\quad&\cdots&\\ 
&\text{Player~2}:\quad&&a_1\quad&&a_3\quad&&\cdots
\end{alignat*}
The two players take turns playing $a_0,a_1,a_2,\ldots\in A$
with the requirement that $(a_0,\ldots,a_n)\in \cT$
for each $n$ (thus, $\cT$ is understood to be
the set of ``rules'' of the game). Player~1 wins if and only if the final sequence
$(a_n)_{n\in\N}$ belongs to $\cX$. A \emph{strategy for Player~1} is a map
$\varphi:A^{<\omega}\to A$, with Player~1 playing $a_0=\varphi(\varnothing)$,
followed by $a_2=\varphi\big((a_1)\big)$, $a_4=\varphi\big( (a_1,a_3)\big)$,
etc., in response to Player~2 playing $a_1$, $a_3$, etc.
The notion of a strategy for Player~2 is 
defined similarly. A strategy is called \emph{winning} if the relevant player
always wins the game while following her strategy.
The game $G(\cT,\cX)$ is \emph{determined} if one of the two 
players has a winning strategy. We equip $A^\omega$ with the product topology of the 
discrete topology on $A$, and we equip $[\cT]$ with the relative topology.
By Martin's theorem, if $\cX$ is a Borel subset of $[\cT]$, then $G(\cT,\cX)$ is determined 
(see \cite[Theorem 20.5]{K}).

\subsection{The main theorem}
\label{sec:pf-main-thm}

We now come to the main result of the paper. 
As in \S\ref{sec:key-lime-pie}, for a given subset $\gB$ of $W$,
we define the trees $T=T_\gB$ and $P=P_\gB$, and we fix $\gb\in\R\cup\{-\infty\}$.
In what follows, we suppose that
$\gell :V\to[\gb,\infty)$ is a $(\gB,\gb)$-target-controlled map.

\begin{theorem}\label{thm:gencomp}
The map $f:P\to[\gb,\infty]$ defined by 
\be\label{eq:complete-function}
f\big((v_n)_{n\in\N_1}\big)\defeq\mathop{\overline{\rm lim}}\limits_{n\in\N_1}\gell (v_n)
\ee
is densely onto and $[\gb,\infty]$-upper semi-Baire class~1 complete.
\end{theorem}

\begin{proof}
Let $(M_n)_{n\in\N_1}$ be defined as in \S\ref{sec:key-lime-pie}.
Let $s\in T$ of length $l_1\defeq{\tt len}(s)$, and assume that
$l_1=0$ or $l_1>M_1$.
Let $g:2^\omega\to [\gb,\infty ]$ be upper semi-Baire class~1.
In the style of the proof of \cite[Theorem 5.1]{Solecki},
we consider the following game, which is of the type $G(\cT,\cX)$ if we code both
$2\defeq\{0,1\}$ and $V$ (which is countably infinite) as subsets of $A\defeq\N$: 
\begin{alignat*}{50}
&\text{Player~1}:\quad&x_0\quad&&x_1\quad&&x_2\quad&\cdots&\\ 
&\text{Player~2}:\quad&&v_1\quad&&v_2\quad&&v_3\quad&&\cdots
\end{alignat*}
with each $x_n\in\{0,1\}$ and each $v_n\in V$; hence the game
produces two sequences $x\defeq(x_n)_{n\in\N}\in 2^\omega$ and
$v\defeq(v_n)_{n\in\N_1}\in V^\omega$.
The only rule of the game is that each $v|_m\defeq(v_1,\cdots,v_m)$ must lie in $T_s$.
Player~1 wins the game if and only if $g(x)\not= f(v)$.

Since $f,g$ are Borel maps, the collection $\cX$ of winning outcomes for
Player~1 is Borel, so the game is determined by Martin's theorem.
A winning strategy for Player~2 induces a continuous function
$\phi:2^\omega\to P_s$ such that $g=f\circ\phi$.
Let us prove that Player~1 cannot have a winning strategy.

Suppose, on the contrary, that Player~1 has a winning strategy, i.e., that
there is a continuous function $\psi:P_s\to 2^\omega$ such that
$g(\psi(v))\ne f(v)$ for every $v\in P_s$.
Note that $g\circ\psi :P_s\to [\gb,\infty ]$ is upper semi-Baire class~1,
and for any increasing homeomorphism $h:[\gb,\infty ]\to [0,1]$, the map
$c\defeq h\circ g\circ\psi :P_s\to [0,1]$ is also upper semi-Baire class~1. By \cite[Theorem 2.1]{Elekes}, there is a map
$\tilde u:T_s\to\R$ such that
$$
\forall\,v\in P_s:\quad
c(v)=\mathop{\overline{\rm lim}}\limits_{m\in\N_1}\tilde u(v|_m).
$$
Define $u':T_s\to [0,1)$ by 
$$
\forall\,t\in T_s:\quad
u'(t)\defeq\begin{cases}
\tilde u(t)&\quad\hbox{if $\tilde u(t)\in [0,1)$},\\
0&\quad\hbox{if $\tilde u(t)<0$},\\
1-2^{-{\tt len}(t)}&\quad\hbox{if $\tilde u(t)\ge 1$}.
\end{cases}
$$
Note that 
$$
\forall\,v\in P_s:\quad
c(v)=\mathop{\overline{\rm lim}}\limits_{m\in\N_1}u'(v|_m).
$$
Finally, defining $u:T_s\to[\gb,\infty)$ by $u\defeq h^{-1}\circ u'$, we see that
$$
\forall\,v\in P_s:\quad
g\big(\psi(v))=h^{-1}\big(c(v)\big)
=\mathop{\overline{\rm lim}}\limits_{m\in\N_1}u(v|_m).
$$
On the other hand, Lemma \ref{lem:tre} provides $v\in P_s$ such that 
$$
g\big(\psi(v))=\mathop{\overline{\rm lim}}\limits_{m\in\N_1}u(v|_m)=
\mathop{\overline{\rm lim}}\limits_{n\in\N_1}\gell(v_n)=f(v),
$$
which is the desired contradiction.

Note that the function $\overline{L}:[\gb,\infty )^\omega\to[\gb,\infty]$ given by 
$$
\forall\,x=(x_n)_{n\in\N_1}:\quad
\overline{L}(x)\defeq\mathop{\overline{\rm lim}}\limits_{n\in\N_1}x_n
$$ 
is upper semi-Baire class~1. Indeed, for any $\ttt\in\R$ we have 
$$
\mathop{\overline{\rm lim}}\limits_{n\in\N_1}x_n<\ttt
\quad\Longleftrightarrow\quad
\exists\,m,k\in\N,\,\forall\,n\ge m:\quad x_n\le \ttt-2^{-k},
$$
and the latter is a $\borels{0}{2}$ condition. Noting that $P$ is Polish (for it
is a closed subset of the Polish space $V^\omega$), and $f$ is the composition
of the continuous map $((v_n)_{n\in\N_1})\mapsto (\gell(v_n))_{n\in\N_1}$
with $\overline{L}$, it follows that $f$ is upper semi-Baire class~1.  Taking $s=\varnothing$, we deduce that $f$ is $[\gb,\infty ]$-upper semi-Baire class~1 complete.

To finish the proof, let $\ttt\in [\gb,\infty ]$,
and let $O$ be an arbitrary nonempty open subset of $P$.
Fix an element $s\in T$ such that ${\tt len}(s)>M_1$ and $P_s\subset O$.
Since $[\gb,\infty]$ is a nonempty metrizable compact space,
\cite[Theorem 4.18]{K} yields a continuous surjection 
$g:2^\omega\to[\gb,\infty]$; in particular, $g$ is upper semi-Baire class~1.
By the above argument, there is a continuous function $\phi:2^\omega\to P_s$
such that $g=f\circ\phi$.
Now, let $x\in 2^\omega$ with 
$g(x)=\ttt$. Then $\phi (x)\in P_s\subset O$ and $f(\phi (x)) =\ttt$, 
which shows that $f$ is densely onto.
\end{proof}

\medskip

\begin{proof}[Proof of Theorems~\ref{thm:densely-onto}
and \ref{thm:main3}]
Let $\ell:V\to[1,\infty)$ be the map given by
\eqref{eq:ell-defn}, and let
$$
B\defeq\big\{\big((x_1,y_1),(x_2,y_2)\big)\in V^2:x_2=y_1
\text{~and~}y_2=ay_1+x_1\text{~for some~}a\in\N_1\big\}.
$$
As noted previously, $\ell$ is $(B,1)$-target-controlled 
in view of Lemma~\ref{lem:targeting}.

To facilitate the proofs, consider the following commutative diagram:
\begin{center}\begin{tikzcd}[arrows={-Stealth}]
    \RQ \arrow{r}{\phi}\arrow{dd}[swap]{\tau}
    \arrow[rrr, "{\lambda}"{anchor=south}, bend left]& 
    \Z\times\N_1^\omega \arrow{r}{\pi} &
    \N_1^\omega \arrow{r}{\Phi} &
    P \arrow{ddlll}{f}\\ 
    & & & \\
    {[1,\infty]} & & & 2^\omega \arrow{uu}[swap]{\varphi} \arrow{lll}[swap]{g}  \\
\end{tikzcd}\end{center}
\vskip-.25in
\noindent
In what follows, we describe the various maps shown in the diagram and study
the interplay between them.
First of all, we have the type function $\tau:\RQ\to[1,\infty]$, which is
the subject of the present paper.
Next, the map $\phi:\RQ\to\Z\times\N_1^\omega$ given by
$$
\alpha=[a_0;a_1,a_2,\ldots]\mapsto\phi(\alpha)\defeq(a_n)_{n\in\N}
$$
is the homeomorphism given in Proposition \ref{prop:homeomorphism}.
To every $(a_n)_{n\in\N}$
we associate a sequence $(k_n)_{n\in\N}$ of positive integers using
the definition \eqref{eq:kn-recursion} (so the $k_n$ are the denominators of
the continued fractions for $\alpha$). Then,
defines a sequence $(v_n)_{n\in\N_1}$ of vertices in $V$ as follows:
$$
\forall\,n\in\N_1:\quad v_n\defeq(k_n,k_{n+1}).
$$
Recall that in \S\ref{sec:fraction-trails} we have seen that
$$
\tau(\alpha)
=\mathop{\overline{\rm lim}}\limits_{n\to\infty}\frac{\log k_{n+1}}{\log k_n}
=\mathop{\overline{\rm lim}}\limits_{n\to\infty}\ell(v_n).
$$
For each $n\in\N_1$ we have $v_n\squig{a_{n+2}} v_{n+1}$, which implies
$(v_n,v_{n+1})\in B$; thus we deduce that $(v_n)_{n\in\N_1}$ lies in $P$. We let 
$f:P\to [1,\infty]$ be the map given by
$$
\forall\,v=(v_n)_{n\in\N_1}\in P:\quad
f(v)\defeq\mathop{\overline{\rm lim}}
\limits_{n\in\N_1}\ell(v_n)=\tau(\alpha).
$$
Next, observe that the map $\Phi:\N_1^\omega\to P$ given by 
$$
\forall\,a=(a_n)_{n\in\N_1}\in \N_1^\omega:\quad
\Phi(a)\defeq(v_n)_{n\in\N_1}
$$ 
is continuous. Moreover, setting $a_1\defeq k_1$ and
$a_{n+2}\defeq(k_{n+2}-k_n)/k_{n+1}$ for all $n\in\N$ and 
$(v_n)_{n\in\N_1}\in P$, it is clear that the inverse map
$\Phi^{-1}:P\to\N_1^\omega$ given by 
$$
\forall\,v=(v_n)_{n\in\N_1}\in P:\quad
\Phi^{-1}(v)\defeq(a_n)_{n\in\N_1}
$$ 
is also continuous. Therefore, $\Phi$ is a homeomorphism of $\N_1^\omega$
onto $P$. Finally, we let $\pi:\Z\times\N_1^\omega\to\N_1^\omega$ be the
projection given by $\pi((a_n)_{n\in\N})\defeq (a_n)_{n\in\N_1}$.

Let $\lambda:\RQ\to P$ be the composition $\Phi\circ\pi\circ\phi$; it is
a continuous map. Note that $\tau=f\circ\lambda$. 
By Theorem~\ref{thm:gencomp}, $f$ is densely onto, and so for any $\ttt\in[1,\infty]$:\dalign{
f^{-1}(\{\ttt\})\text{~is dense in~}P
&\quad\Longrightarrow\quad
(f\circ\Phi)^{-1}(\{\ttt\} )\text{~is dense in~}\N_1^\omega\\
&\quad\Longrightarrow\quad
(f\circ\Phi\circ\pi)^{-1}(\{\ttt\})\text{~is dense in~}\Z\times\N_1^\omega\\
&\quad\Longrightarrow\quad
\tau^{-1}(\{\ttt\})=(f\circ\lambda)^{-1}(\{\ttt\})\text{~is dense in~}\RQ,
}
which completes our third proof of Theorem~\ref{thm:densely-onto}.

We turn to the proof of Theorem~\ref{thm:main3}.
By Theorem~\ref{thm:gencomp}, $f$ is $[1,\infty]$-upper semi-Baire class~1 complete.
If $g:2^\omega\to [1,\infty]$ is also upper semi-Baire class~1, then 
there is a continuous map $\varphi:2^\omega\to P$
for which $g=f\circ\varphi$ (see the diagram).
Let $\rho:2^\omega\to\RQ$ be defined by
$$
\forall\,x\in 2^\omega:\quad 
\rho(x)\defeq\phi^{-1}\big(0,\Phi^{-1}\big(\varphi(x)\big)\big).
$$
Then $\rho$ is continuous, and
$$
\forall\,x\in 2^\omega:\quad
\tau\big(\rho(x)\big)=f\big(\varphi(x)\big)=g(x).
$$
The theorem follows.
\end{proof}

\newpage{\large\section{Complexity and Baire category}
\label{sec:complexity}}

In this section, we give general conditions ensuring the descriptive set
theoretic property of the type mentioned in Theorem~\ref{thm:main2}. 
We also provide Baire category properties.

Recall that a subset of a topological space is said to be \emph{nowhere dense} if its
closure has an empty interior,
\emph{meager} if it is a countable union of nowhere dense
sets, and \emph{comeager} if its complement is meager.
Meagerness is a very useful notion 
of smallness in any completely metrizable space (cf.\ \cite[\S8.B]{K}) since
the class of meager sets is a $\sigma$-\emph{ideal}, i.e., it is closed under taking 
subsets and countable unions. A property that holds for a comeager set of points 
intuitively holds topologically ``almost everywhere'' in such spaces.
Of course, a set can be large in the sense of topology and small in the sense of
measure. Consider, for example, the set of Liouville numbers. The proof of
Theorem \ref{thm:main2} shows that the Liouville numbers form a comeager subset
of $\RQ$, whereas the same set is well known to have Lebesgue measure zero;
see Oxtoby~\cite[\S2]{Oxtoby} (we also refer the reader to \cite{Oxtoby} for
general background on measure and topology).
 
\bigskip
 
\begin{theorem}\label{thm:gen++}
The function $f:P\to [\gb,\infty ]$ defined by \eqref{eq:complete-function}
is upper semi-Baire class~1 and Baire class 2, but it is not Baire class~1.
More precisely, $f$ has the properties outlined in the following
table.\medskip

\centerline{\scalebox{0.9}{$
\begin{tabular}{|c|c|c|c|}
\hline\vphantom{\Big|}
 & $\ttt=\gb$ & $\ttt\in(\gb,\infty )$ & $\ttt=\infty$ \\
\hline\vphantom{\Big|}
$v:f(v)<\ttt$ & $\varnothing$ & meager $\borels{0}{2}$-complete & meager $\borels{0}{2}$-complete\\
\hline\vphantom{\Big|}
$v:f(v)\ge \ttt$ & $P$ & comeager $\borelp{0}{2}$-complete & comeager $\borelp{0}{2}$-complete\\
\hline\vphantom{\Big|}
$v:f(v)>\ttt$ & comeager $\borels{0}{3}$-complete & comeager $\borels{0}{3}$-complete & $\varnothing$ \\
\hline\vphantom{\Big|}
$v:f(v)\le \ttt$ & meager $\borelp{0}{3}$-complete & meager $\borelp{0}{3}$-complete & $P$\\
\hline\vphantom{\Big|}
$v:f(v)=\ttt$ & meager $\borelp{0}{3}$-complete & meager $\borelp{0}{3}$-complete & comeager $\borelp{0}{2}$-complete\\
\hline
\end{tabular}$
}}
\end{theorem}

\begin{proof}
First, note that $f$ is upper semi-Baire class~1, by Theorem \ref{thm:gencomp}.
This implies that $f$ is Baire class 2 because
\begin{alignat*}{2}
f(v)<\infty&\quad\Longleftrightarrow\quad
\exists\,n\in\N:\quad f(v)<n
&&\qquad(\text{$\borels{0}{2}$ condition}),\\
f(v)>\ttt&\quad\Longleftrightarrow\quad
\exists\,k\in\N:\quad f(v)\ge \ttt+2^{-k}
&&\qquad(\text{$\borels{0}{3}$ condition}),\\
f(v)>-\infty&\quad\Longleftrightarrow\quad
\exists\,n\in\Z:\quad f(v)\ge n
&&\qquad(\text{$\borels{0}{3}$ condition}).
\end{alignat*}
On the other hand, we show below that the condition $f(v)>\ttt$ is not $\borelp{0}{3}$
for any real $\ttt\ge \gb$ (hence not a $\borels{0}{2}$ condition),
and thus $f$ is not Baire class~1.

As above, the set $f^{-1}(\{\infty\})$ is in $\borelp{0}{2}$,
and by Theorem~\ref{thm:gencomp}, the map $f$ is densely onto; therefore, 
$f^{-1}(\{\infty\})$ is a dense $\borelp{0}{2}$, and thus it is a countable
intersection of dense open sets. The complement of $f^{-1}(\{\infty\})$ is a
countable union of closed sets with an empty interior,
and so $f^{-1}(\{\infty\})$ is comeager, a fact that implies all
of the meagerness and comeagerness properties listed in the table.

Again, since $f$ is densely onto, both $f^{-1}(\{\infty\})$ and its complement
are dense in $P$. As in the proof of Lemma~\ref{lem:Cantor-pi02}, we deduce that
$f^{-1}(\{\infty\})$ is $\borelp{0}{2}$-complete.
Using similar arguments, we find that the set $f^{-1}([-\infty,\ttt))$ is
$\borels{0}{2}$-complete for any $\ttt\in(\gb,\infty)$,
This proves our assertions for the rows $v:f(v)<\ttt$ and $v:f(v)\ge\ttt$ in the
table, as well as the column $\ttt=\infty$.

From now on, let $\ttt\in [\gb,\infty)$ be fixed.
We have seen that $f^{-1}([-\infty,\ttt])$ is $\borelp{0}{3}$
and $f^{-1}([\ttt,\infty])$ is $\borelp{0}{2}$,
which implies that $f^{-1}(\{\ttt\} )$ is $\borelp{0}{3}$.
To finish the proof, we need to prove that $f^{-1}([-\infty,\ttt])$
and $f^{-1}(\{t\})$ are $\borelp{0}{3}$-complete.

In view of Lemma~\ref{lem:diverges-pi03}, it is enough to show 
that $\cN_\infty\le_Wf^{-1}([-\infty,\ttt])$ and
$\cN_\infty\le_Wf^{-1}(\{\ttt\})$, and for this, it suffices to find
a continuous map
$$
F:\N_1^\omega\to P
$$
with the following property:
\be\label{eq:cat-tails}
\beta\in\cN_\infty\quad\Longleftrightarrow\quad
f(F(\beta))=\ttt\quad\Longleftrightarrow\quad f(F(\beta))\le \ttt.
\ee
One such map $F$ can be constructed as follows.

Fix once and for all two
strictly decreasing sequences $(\ttt_j)_{j\in\N}$ and
$(\eps_j)_{j\in\N}$ of real numbers with
\be\label{eq:nueps-limits}
\lim\limits_{j\to\infty}\ttt_j=\ttt
\mand
\lim\limits_{j\to\infty}\eps_j=0.
\ee
We further assume that $\eps_j<1$ for all $j\in\N$, and that the intervals
$$
\forall j\in\N:\quad
U_j\defeq [\ttt_j,\ttt_j+\eps_j)
$$
are pairwise disjoint (equivalently, $\ttt_i+\eps_i\le \ttt_j$ for all $i>j$).
Let $(y_j)_{j\in\N}$ be a (strictly increasing) sequence of positive integers with the property that
$$
\forall j\in\N:\quad
y_j\ge M_{\gB,\gb}(\eps_j).
$$
This definition is motivated by the fact that $\gell$ is $(\gB,\gb)$-target-controlled, 
which implies that if $v=(x,y)$ lies in $V$ and $y>y_j$, then there is a vertex
$w\in V$ such that $(v,w)\in \gB$ and $\gell(w)\in U_j$.

To define $F$, let $\beta\in\N_1^\omega$ be given.
Regardless of the value $\beta(0)$, let
$$
k_0\defeq 1,\qquad
k_1\defeq y_1+1,\mand
v_0\defeq(k_0,k_1).
$$
Using induction on $n$, we now construct a sequence $(k_n)_{n\in\N}$
of positive integers and a sequence $(v_n)_{n\in\N_1}$ in $V$.
Indeed, suppose that $v_{n-1}=(k_{n-1},k_n)\in V$ 
has already been defined for some $n\in\N_1$. Let $j(n)$ be the largest
integer satisfying both inequalities
\be\label{eq:j(n)bounds}
j(n)\le\beta(n)\mand y_{j(n)}<k_n.
\ee
Note that $j(n)\ge 1$. Since $\gell$ is $(\gB,\gb)$-target-controlled,
there exists a vertex $v_n=(k_n,k_{n+1})\in V$ such that
\be\label{eq:cough}
(v_{n-1},v_n)\in \gB
\mand
\gell(v_n)\in U_{j(n)},
\ee
which completes the induction. Note that the resulting sequence $(k_n)_{n\in\N}$
is strictly increasing. We define $F:\N_1^\omega\to P$ according to the rule
$$
\beta\mapsto F(\beta)\defeq (v_n)_{n\in\N_1}.
$$
The map $F$ is continuous since, for every $n\in\N$, the choice of
$v_n$ is made using only the values $\beta(m)$ with $m\le n$.

Next, let $\beta\in\N_1^\omega$ be fixed. With the notation above,
let $v\defeq F(\beta)$. By \eqref{eq:complete-function} we have
\be\label{eq:its-a-miracle}
f(v)=\mathop{\overline{\rm lim}}\limits_{n\in\N_1}\gell (v_n),
\ee
and each $\gell(v_n)$ lies in the interval
$U_{j(n)}=[\ttt_{j(n)},\ttt_{j(n)}+\eps_{j(n)})$. 
Putting everything together, we have
\begin{alignat*}{2}
\beta\in\cN_\infty&\quad\Longrightarrow\quad
\lim\limits_{n\to\infty}\beta(n)=\infty
&&\quad\text{(definition of $\cN_\infty$)}\\
&\quad\Longrightarrow\quad
\lim\limits_{n\to\infty}j(n)=\infty
&&\quad\text{(see \eqref{eq:j(n)bounds})}\\
&\quad\Longrightarrow\quad
\lim\limits_{n\to\infty}\ttt_{j(n)}=\ttt\text{~and~}
\lim\limits_{n\to\infty}\eps_{j(n)}=0
&&\quad\text{(see \eqref{eq:nueps-limits})}\\
&\quad\Longrightarrow\quad
\lim\limits_{n\to\infty}\gell(v_n)=\ttt
&&\quad\text{(definition of $U_{j(n)}$)}\\
&\quad\Longrightarrow\quad
f(v)=\ttt
&&\quad\text{(see \eqref{eq:its-a-miracle})}.
\end{alignat*}
On the other hand,
$$
\beta\not\in\cN_\infty\quad\Longrightarrow\quad
\mathop{\underline{\rm lim}}_{n\to\infty}\beta(n)<\infty
\quad\Longrightarrow\quad
\mathop{\underline{\rm lim}}\limits_{n\to\infty}j(n)<\infty.
$$
The last condition implies that there is a positive integer $j_0$
such that $j(n)=j_0$ for infinitely many $n$; therefore,
$$
f(v)=\mathop{\overline{\rm lim}}\limits_{n\to\infty}\gell(v_n)
\ge \ttt_{j_0}>\ttt.
$$
Since $v=F(\beta)$, we have established \eqref{eq:cat-tails},
and we are done.
\end{proof}

We are now ready to prove Theorem~\ref{thm:main2}.

\begin{proof}[Proof of Theorem~\ref{thm:main2}]
As in the proof of Theorems~\ref{thm:densely-onto} and~\ref{thm:main3}
(see \S\ref{sec:pf-main-thm}), we have the commutative diagram:
\begin{center}\begin{tikzcd}[arrows={-Stealth}]
    \RQ \arrow{r}{\phi}\arrow{ddr}[swap]{\tau} & 
    \Z\times\N_1^\omega\,\mathop{\xrightarrow{\hspace*{1cm}}}
    \limits^\pi\,\N_1^\omega \arrow{r}{\Phi} &
    P \arrow{ddl}{f}\\ 
    & & \\
    & {[1,\infty]} & \\
\end{tikzcd}\end{center}
\vskip-.25in
\noindent
For every subset $S\subset\N_1^{\N_1}$, we have $\Z\times S\le_WS$
with witness $(z,a)\mapsto a$, and $S\le_W\Z\times S$ with witness $a\mapsto (0,a)$.
Since $\phi$ and $\Phi$ are homeomorphisms,
the topological properties of $f$ given by Theorem~\ref{thm:gen++} imply
the properties of $\tau$ stated in Theorem~\ref{thm:main2}.
\end{proof}

\begin{remark*}
In Theorem \ref{thm:gen++}, one can replace $\overline\R$ with a dense complete
lattice whose order topology is metrizable. However, such an ordered topological
space is isomorphic to $\overline\R$, so this is not really a generalization.
\end{remark*}

\appendix

\newpage{\large\section{Some equivalent definitions of type}}

The type function can be defined in various ways, and our aim is
to show that the definition in our abstract is consistent with the
definition given in \S\ref{sec:irrtypdef}.
Here, we define type functions $\tau_1$, $\tau_2$, and $\tau_3$ in 
three different ways, showing that $\tau_1=\tau_2=\tau_3$.

In the abstract, we define $\tau\defeq\tau_1$ with
$$
\tau_1(\alpha)\defeq\inf\{t\in\R:|\alpha-h/k|<k^{-t-1}\text{~for only
finitely many~}(h,k)\in\Z\times\N_1\}.
$$
On the other hand, in \S\ref{sec:irrtypdef} we define $\tau\defeq\tau_2$
with
$$
\tau_2(\alpha):=\sup\big\{\theta\in\R:
\mathop{\underline{\rm lim}}\limits_{k\in\N}
~k^\theta\nearint{k\alpha}=0\big\}.
$$
Finally, a useful intermediate definition is $\tau\defeq\tau_3$ with
$$
\tau_3(\alpha):=\sup\big\{\theta\in\R:
\mathop{\underline{\rm lim}}\limits_{k\in\N}
~k^\theta\nearint{k\alpha}<1\big\}.
$$

We have trivially
$$
\mathop{\underline{\rm lim}}\limits_{k\in\N}
~k^\theta\nearint{k\alpha}=0
\quad\Longrightarrow\quad
\mathop{\underline{\rm lim}}\limits_{k\in\N}
~k^\theta\nearint{k\alpha}<1,
$$
and in the opposite direction,
$$
\forall\eps>0:\quad\mathop{\underline{\rm lim}}\limits_{k\in\N}
~k^\theta\nearint{k\alpha}<1
\quad\Longrightarrow\quad
\mathop{\underline{\rm lim}}\limits_{k\in\N}
~k^{\theta-\eps}\nearint{k\alpha}=0.
$$
These implications imply that $\tau_2=\tau_3$.

Next, we show $\tau_1=\tau_3$. If $\tau_3(\alpha)<\infty$, then
from the chain of implications
\dalign{
t>\tau_3(\alpha)&\quad\Longrightarrow\quad
\mathop{\underline{\rm lim}}\limits_{k\in\N}
~k^t\nearint{k\alpha}\ge 1\\
&\quad\Longrightarrow\quad
\exists\,k_0~\forall\,k\ge k_0:\quad k^t\nearint{k\alpha}\ge 1\\
&\quad\Longrightarrow\quad
\exists\,k_0~\forall\,k\ge k_0:\quad k^t\min_{h\in\Z}|k\alpha-h|\ge 1\\
&\quad\Longrightarrow\quad
\exists\,k_0~\forall\,k\ge k_0~\forall\,h:\quad |\alpha-h/k|\ge k^{-t-1}\\
&\quad\Longrightarrow\quad t\ge\tau_1(\alpha),
}
thus $\tau_1(\alpha)$ is also finite, and
\be\label{eq:tnt}
\tau_1(\alpha)\le\tau_3(\alpha).
\ee
Note that \eqref{eq:tnt} is trivial when $\tau_3(\alpha)=\infty$.
In the other direction, we have
\dalign{
t<\tau_3(\alpha)&\quad\Longrightarrow\quad
\forall\,\eps>0:\quad\mathop{\underline{\rm lim}}\limits_{k\in\N}
~k^{t-\eps}\nearint{k\alpha}<1\\
&\quad\Longrightarrow\quad
\forall\,\eps>0~\exists^\infty k:\quad k^{t-\eps}\nearint{k\alpha}<1\\
&\quad\Longrightarrow\quad
\forall\,\eps>0~\exists^\infty k:\quad k^{t-\eps}\min_{h\in\Z}|k\alpha-h|<1\\
&\quad\Longrightarrow\quad
\forall\,\eps>0~\exists^\infty k~\exists\,h:\quad |\alpha-h/k|<k^{-t-\eps-1}\\
&\quad\Longrightarrow\quad
\forall\,\eps>0:\quad t-\eps\le\tau_1(\alpha).
}
If $\tau_3(\alpha)=\infty$, then taking $t\to\infty$ we deduce that
$\tau_1(\alpha)=\infty$ as well; in particular,
\be\label{eq:tnt2}
\tau_3(\alpha)\le\tau_1(\alpha).
\ee
If $\tau_3(\alpha)<\infty$, we see that
$\tau_3(\alpha)\le\tau_1(\alpha)+\eps$ for every $\eps>0$, which implies
\eqref{eq:tnt2} in this case. Finally, combining
\eqref{eq:tnt} and \eqref{eq:tnt2}, we have $\tau_1=\tau_3$.

\end{document}